\documentclass[10pt]{amsart}
\usepackage{amsmath,amssymb,mathrsfs,color}

\def\R{{\mathbb R}}
\newtheorem{lemma}{Lemma}[section]
\newtheorem{theorem}[lemma]{Theorem}
\newtheorem{remark}[lemma]{Remark}
\newtheorem{prop}[lemma]{Proposition}
\newtheorem{coro}[lemma]{Corollary}
\newtheorem{definition}[lemma]{Definition}
\newtheorem{example}[lemma]{Example}
\allowdisplaybreaks
\numberwithin{equation}{section}

\begin{document}

\title[Poisson stable solutions for SDEs]
{Periodic, Quasi-Periodic, Almost Periodic, Almost Automorphic, Birkhoff Recurrent and
Poisson Stable Solutions for Stochastic Differential Equations}

\author{David Cheban}
\address{D. Cheban\footnote{The permanent address of D. Cheban is: State University of Moldova, Faculty of Mathematics and
Informatics, Department of Mathematics, A. Mateevich Street 60, MD--2009 Chi\c{s}in\u{a}u, Moldova}: School of Mathematical Sciences,
Dalian University of Technology, Dalian 116024, P. R. China}
\email{cheban@usm.md; davidcheban@yahoo.com}

\author{Zhenxin Liu}
\address{Z. Liu (Corresponding author): School of Mathematical Sciences,
Dalian University of Technology, Dalian 116024, P. R. China}
\email{zxliu@dlut.edu.cn}


\date{February 8, 2017}


\begin{abstract}

The paper is dedicated to studying the problem of Poisson stability
(in particular stationarity, periodicity, quasi-periodicity, Bohr almost
periodicity, Bohr almost automorphy, Birkhoff recurrence, 
almost recurrence in the sense of Bebutov, Levitan almost periodicity, pseudo-periodicity,
pseudo-recurrence, Poisson stability) of
solutions for semi-linear stochastic equation
$$
dx(t)=(Ax(t)+f(t,x(t)))dt +g(t,x(t))dW(t)\quad (*)
$$
with exponentially stable linear operator $A$ and Poisson stable in
time coefficients $f$ and $g$. We prove that if the functions $f$
and $g$ are appropriately ``small", then equation $(*)$ admits at
least one solution which has the same character of recurrence as the
functions $f$ and $g$.
\end{abstract}

\maketitle

\section{Introduction}

A continuous function $\varphi$ defined on real line $\mathbb R$
with values in a metric space $(X,\rho)$ is said to be Poisson
stable \cite{Sch72,scher75,Sch85,sib} in the positive
(respectively, negative) direction if there is a sequence
$\{t_n\}\subset \mathbb R$ with $t_n\to +\infty$ (respectively,
$ t_n\to -\infty$) such that $\varphi(t+t_n)\to \varphi(t)$
uniformly with respect to $t$ on every compact interval $[-l,l]$
($l>0$) as $n\to \infty$. If $\varphi$ is Poisson stable in both
directions, then it is called Poisson stable.

One considers \cite{Sch72,scher75,Sch85,sib} the following classes of
Poisson stable functions: stationary (respectively, periodic,
quasi-periodic \cite{Bol_1893,Bol_1906}, Bohr almost periodic
\cite{Bohr_1923,Bohr_I,Bohr_II,Bohr_III,Bohr_I1947},
 almost automorphic \cite{Boh_1955,B62,Wee_1965}, Birkhoff recurrent
\cite{Bir_1927}, Levitan almost periodic \cite{Lev_1938,Lev_1953},
almost recurrent in the sense of Bebutov \cite{Beb_1940,sib}, pseudo-periodic \cite[p.32]{Bohr_I1947},
pseudo-recurrent \cite{Sch68,Sch72,sib},  Poisson stable \cite{Sch72,sib}) functions, among
others.

In his works \cite{Sch72,scher75,Sch85,scher-ch,ShFa_1977}, B. A. Shcherbakov
systematically studied the problem of existence of Poisson stable solutions of the equation
\begin{equation}\label{eqI1}
x'=f(t,x), \quad x\in\mathfrak B
\end{equation}
with the  right-hand side $f$ Poisson stable in $t\in\mathbb R$
uniformly with respect to $x$ on every compact subset
from $\mathfrak B$, where $\mathfrak B$ is a Banach space.

To study this problem, B. A. Shcherbakov established
a method (principle) of comparability of functions by character of their recurrence.
Using his method B. A. Shcherbakov studied different classes
of equations of the form (\ref{eqI1}) for which he gave conditions of
existence at least one (or exactly one) solution with the same
character of recurrence as right-hand side $f$. He named this type of solution  {\em comparable}
(respectively, {\em uniformly comparable}) {\em solution} of equation (\ref{eqI1}).

Later the works of B. A. Shcherbakov were extended and generalized by many authors:
I. Bronshtein \cite[ChIV]{bro75}, T. Caraballo and
D. Cheban \cite{CC_I,CC_II,CC_2011,CC_2013_1}, D. Cheban \cite{Che_2008,Che_2009},
D. Cheban and C. Mammana \cite{CM_2005}, D. Cheban and B. Schmalfuss \cite{CS_2008}, and others.

In this paper, we try to extend and generalize Shcherbakov's
ideas and methods to study the Poisson stability of solutions for stochastic
differential equations
\begin{equation*}\label{eqI2}
dx(t)=f(t,x(t))dt + g(t,x(t))dW(t),
\end{equation*}
where $f$ and $g$ are Poisson stable functions in $t$.

Note that this problem was studied before only for periodic, Bohr almost
periodic and Bochner almost automorphic equations: see, e.g. \cite{DTr,DT,Ich,Kh,Mo} for periodic
equations, \cite{AT,BD,DT,Hal,KMR,LW_2016,T,WL} for Bohr almost periodic
equations and \cite{CZ,FL,LS,X} for Bochner almost automorphic equations,
and references therein. It should be pointed out that
either Bohr almost periodic or Bochner almost automorphic solutions can be only in distribution sense
instead of in square-mean sense, see \cite{KMR,LS} for details.
We consider in our present work the general problem of Poisson stability
for all classes listed above.

This paper is organized as follows.

In the second section we collect some known notions and facts.
Namely we present the definitions of all important classes of Poisson stable functions and
their basic properties. We also give a short survey
of Shcherbakov's results on comparability of functions by character of their recurrence.

The third section is dedicated to studying Poisson stable solutions for the linear equation
\begin{equation}\label{eqI3}
dx(t)=(Ax(t) + f(t))dt + g(t)dW(t)
\end{equation}
with exponentially stable linear operator $A$ (generally unbounded). The main result of this section (Theorem \ref{thLS11_2})
states that equation (\ref{eqI3}) with bounded coefficients $f$ and $g$ admits a unique bounded solution $\varphi$ which has the same character of
recurrence in distribution as $f$ and $g$.

In the fourth section we study the problem of Poisson stability for the semi-linear equation
\begin{equation}\label{eqI5}
dx(t)=(Ax(t)+F(t,x(t)))dt + G(t,x(t))dW(t).
\end{equation}
We prove (Theorem \ref{t3.4.4}) that the equation
(\ref{eqI5}) has a unique bounded solution $\xi$ which has the same character of recurrence as the functions $F$ and $G$.

The fifth section is dedicated to studying the dissipativity (Theorem \ref{thC1}) and the convergence (Theorem \ref{thC2}) for equation
(\ref{eqI5}).

In the last section, we give some applications of our theoretical results.

\section{Preliminaries}

\subsection{The space $C(\mathbb R, X)$}
Let $(X,\rho)$ be a complete metric space. Denote by $C(\mathbb
R,X)$ the space of all continuous functions $\varphi :\mathbb R
\to X$ equipped with the distance
\begin{equation*}\label{eqD1}
d(\varphi,\psi):=\sup\limits_{L>0}\min\{\max\limits_{|t|\le
L}\rho(\varphi(t),\psi(t)),L^{-1}\}.
\end{equation*}
The space $(C(\mathbb R,X),d)$ is a complete metric space (see, for
example, \cite[ChI]{Sch72},\cite{Sch85,sib}). Throughout the paper,
convergence in $C(\mathbb R, X)$ means the convergence with respect
to this metric $d$ if not specified otherwise.

\begin{lemma}\label{l1} {\rm(\cite[ChI]{Sch72},\cite{Sch85,sib})} The following statements
hold:
\begin{enumerate}
\item $d(\varphi,\psi) = \varepsilon$ if and only if
$$
\max\limits_{|t|\le \varepsilon^{-1}}\rho(\varphi(t),\psi(t))=\varepsilon ;
$$
\item $d(\varphi,\psi)<\varepsilon$ if and only if
$$
\max\limits_{|t|\le \varepsilon^{-1}}\rho(\varphi(t),\psi(t))<\varepsilon ;
$$
\item $d(\varphi,\psi)>\varepsilon$ if and only if
$$
\max\limits_{|t|\le \varepsilon^{-1}}\rho(\varphi(t),\psi(t))>\varepsilon .
$$
\end{enumerate}
\end{lemma}

\begin{remark}\label{remD1} \rm
1. The distance $d$ generates on $C(\mathbb R,X)$ the
compact-open topology.

2. The following statements are equivalent:
\begin{enumerate}
\item $d(\varphi_n,\varphi)\to 0$ as $ n\to \infty$;
\item  $\lim\limits_{n\to \infty}\max\limits_{|t|\le
L}\rho(\varphi_n(t),\varphi(t))=0$ for each $L>0$;
\item there exists a sequence $l_n\to +\infty$ such that $\lim\limits_{n\to \infty}\max\limits_{|t|\le
l_n}\rho(\varphi_n(t),\varphi(t))=0$.
\end{enumerate}
\end{remark}

\subsection{Poisson stable functions}

Let us recall the types of Poisson stable functions to be studied in this paper; we refer the reader
to \cite{Sel_71,Sch72,Sch85,sib} for further details and the relations among these types of functions.

\begin{definition} \rm
A function $\varphi\in C(\mathbb R,X)$ is called {\em stationary} (respectively, {\em $\tau$-periodic})
if $\varphi(t)=\varphi(0)$ (respectively, $\varphi(t+\tau)=\varphi(t)$) for all $t\in \mathbb R$.
\end{definition}

\begin{definition} \rm
Let $\varepsilon >0$. A number $\tau \in \mathbb R$ is called {\em $\varepsilon$-almost period} of the function
$\varphi$ if $\rho(\varphi(t+\tau),\varphi(t))<\varepsilon$ for all $t\in\mathbb R$. Denote by $\mathcal T(\varphi,\varepsilon)$
the set of $\varepsilon$-almost periods of $\varphi$.
\end{definition}

\begin{definition} \rm
A function $\varphi \in C(\mathbb R,X)$ is said to be {\em Bohr almost periodic} if the set of
$\varepsilon$-almost periods of $\varphi$
is {\em relatively dense} for each $\varepsilon >0$, i.e. for each $\varepsilon >0$ there exists $l=l(\varepsilon)>0$ such that
$\mathcal T(\varphi,\varepsilon)\cap [a,a+l]\not=\emptyset$ for all $a\in\mathbb R$.
\end{definition}

\begin{definition} \label{ppf}\rm 
A function $\varphi \in C(\mathbb R,X)$ is said to be {\em pseudo-periodic} in the positive
(respectively, negative) direction if for each $\varepsilon >0$ and
$l>0$ there exists a $\varepsilon$-almost period $\tau >l$
(respectively, $\tau <-l$) of the function $\varphi$. The function
$\varphi$ is called pseudo-periodic if it is pseudo-periodic in both
directions.
\end{definition}

\begin{definition}\rm
For given $\varphi\in C(\mathbb R, X)$, denote by $\varphi^h$ the
{\em $h$-translation of $\varphi$}, i.e. $\varphi^h(t)=\varphi(h+t)$
for $t\in\mathbb R$. The {\em hull of $\varphi$}, denoted by
$H(\varphi)$, is the set of all the limits of $\varphi^{h_n}$ in
$C(\mathbb R, X)$, i.e.
\[
H(\varphi):=\{\psi\in C(\mathbb R, X): \psi=\lim_{n\to\infty}
\varphi^{h_n} \hbox{ for some sequence } \{h_n\} \subset \mathbb R\}.
\]
\end{definition}

It is well-known (see, e.g. \cite{Ch2015}) that the mapping $\sigma: \mathbb R\times C(\mathbb R, X)\to C(\mathbb R, X)$ defined by $\sigma(h,\varphi) =\varphi^h$
is a dynamical system, i.e. $\sigma(0,\varphi)=\varphi$, $\sigma(h_1+h_2,\varphi)=\sigma(h_2,\sigma(h_1,\varphi))$ and
the mapping $\sigma$ is continuous. In particular, the mapping $\sigma$ restricted to $\mathbb R\times H(\varphi)$ is a
dynamical system.

\begin{remark} \rm
A function $\varphi \in C(\mathbb R,X)$ is pseudo-periodic in the
positive (respectively, negative) direction if and only if there is
a sequence $t_n\to +\infty$ (respectively, $t_n\to -\infty$) such
that $\varphi^{t_n}$ converges to $\varphi$ uniformly in $t\in
\mathbb R$ as $n\to \infty$.
\end{remark}

\begin{definition} \rm
A number $\tau\in\mathbb R$ is said to be {\em $\varepsilon$-shift} for $\varphi \in C(\mathbb R,X)$ if
$d(\varphi^{\tau},\varphi)<\varepsilon$.
\end{definition}

\begin{definition} \rm
A function $\varphi \in C(\mathbb R,X)$ is called {\em almost recurrent (in the sense of Bebutov)} if for every
$\varepsilon >0$ the set $\{\tau :\ d(\varphi^{\tau},\varphi)<\varepsilon\}$ is relatively dense.
\end{definition}

\begin{definition} \rm
A function $\varphi\in C(\mathbb R,X)$ is called {\em Lagrange stable} if $\{\varphi^{h}:\ h\in \mathbb R\}$
is a relatively compact subset of $C(\mathbb R,X)$.
\end{definition}

\begin{lemma}\label{l2} {\rm(\cite[ChI]{Sch72})} Let $\varphi \in C(\mathbb R,X)$, then the following statements are equivalent:
\begin{enumerate}
\item the function $\varphi$ is Lagrange stable;
\item the function $\varphi$ is uniformly continuous on $\mathbb R$ and its image $\varphi (\mathbb R)$ is a relatively compact subset of $X$.
\end{enumerate}
\end{lemma}

\begin{definition} \rm
A function $\varphi \in C(\mathbb R,X)$ is called {\em Birkhoff recurrent} if
it is almost recurrent and Lagrange stable.
\end{definition}

\begin{definition} \rm
A function $\varphi \in C(\mathbb R,X)$ is called {\em Poisson stable} in the positive (respectively, negative) direction
if for every $\varepsilon >0$ and $l>0$ there exists $\tau >l$ (respectively, $\tau <-l$) such that $d(\varphi^{\tau},\varphi)<\varepsilon$.
The function $\varphi$ is called Poisson stable if it is Poisson stable in both directions.
\end{definition}

In what follows, we denote as well $Y$ a complete metric space.

\begin{definition} \rm
A function $\varphi\in C(\mathbb R,X)$ is called {\em Levitan almost periodic} if there exists
a Bohr almost periodic function $\psi \in C(\mathbb R,Y)$ such that for any $\varepsilon >0$ there exists $\delta =\delta (\varepsilon)>0$
such that $d(\varphi^{\tau},\varphi)<\varepsilon$ for all $\tau \in \mathcal T(\psi,\delta)$, recalling that $\mathcal T(\psi,\delta)$
denotes the set of $\delta$-almost periods of $\psi$.
\end{definition}

\begin{remark} \rm
\begin{enumerate}
\item  Every Bohr almost periodic function is Levitan almost periodic.

\item  The function $\varphi \in C(\mathbb R,\mathbb R)$ defined by equality $\varphi(t)=\dfrac{1}{2+\cos t +\cos \sqrt{2}t}$
is Levitan almost periodic, but it is
not Bohr almost periodic \cite[ChIV]{Lev-Zhi}.
\end{enumerate}
\end{remark}

\begin{definition}\label{defAA1} \rm
A function $\varphi \in C(\mathbb R,X)$ is said to be {\em Bohr almost
automorphic} if it is Levitan almost periodic and Lagrange stable.
\end{definition}

\begin{remark}\label{remAA1} \rm
\begin{enumerate}

\item
The function $\varphi \in C(\mathbb R,X)$ is Bohr almost
automorphic if and only if for any sequence $\{t'_{n}\} \subset
\mathbb R$ there are a subsequence $\{t_n\}$ and some function
$\psi: \mathbb R\to X$ such that
\begin{equation}\label{eqAA1}
\varphi(t+t_n)\to \psi(t)\ \ \mbox{and}\ \ \psi(t-t_n)\to \varphi(t)
\end{equation}
uniformly in $t$ on every compact subset from $\mathbb R$.
Some authors call this later equivalent version  ``compact almost automorphy".

\item In \cite{Wee_1965} Veech introduced a bit weaker version of
Bohr almost automorphy as follows: the function $\varphi \in
C(\mathbb R,X)$ is called Bohr almost automorphic if it is Levitan
almost periodic and $\varphi(\mathbb R)$ is relatively compact.
In what follows, we mean our version when we mention
Bohr almost automorphy.

\item A function $\varphi \in C(\mathbb R,X)$ is said to be {\em Bochner almost automorphic} (see \cite{Boh_1955,B62} for details)
if from every sequence $\{t'_{n}\} \subset \mathbb R$ we can extract
a subsequence $\{t_n\}$ such that the relations in (\ref{eqAA1})
take place pointwise for $t\in \mathbb R$.

\item It is natural to consider almost automorphy in the sense of Bohr since the
solutions of differential equations satisfy this stronger property;
see \cite{SY} for details.
\end{enumerate}
\end{remark}

\begin{lemma}\label{lAA1} Suppose that the function $\varphi \in C(\mathbb R,X)$ is uniformly
continuous on $\mathbb R$ and almost automorphic in the sense of
Bochner. Then it is almost automorphic in the sense of Bohr.
\end{lemma}
\begin{proof}
Suppose that the function $\varphi \in C(\mathbb R,X)$ is almost
automorphic in Bochner's sense and $\{t_n'\}$ is an arbitrary sequence
from $\mathbb R$, then there exists a subsequence $\{t_n\}$ of
$\{t_n'\}$ such that the relations (\ref{eqAA1}) hold for every
$t\in \mathbb R$. If, additionally, the function $\varphi$ is
uniformly continuous on $\mathbb R$, then the convergence in the
relations (\ref{eqAA1}) are uniform with respect to $t$ on every compact
interval $[-l,l]$ ($l>0$). In fact, if we suppose
that this is not true, then there exist $\varepsilon_0>0$, $l_0>0$ and a
subsequence $\{t_{n_k}\}\subseteq \{t_n\}$ such that at least one of the
inequalities
\begin{equation}\label{eqAA2}
\max\limits_{|t|\le l_0}\rho(\varphi(t+t_{k_n}),\psi(t))\ge
\varepsilon_0
\end{equation}
and
\begin{equation}\label{eqAA3}
\max\limits_{|t|\le l_0}\rho(\psi(t-t_{k_n}),\varphi(t))\ge
\varepsilon_0
\end{equation}
takes place.

Since the function $\varphi$ is almost automorphic in the sense of
Bochner, the closure $\overline{\varphi(\mathbb R)}$ of its image is a
compact subset of $X$. Since $\varphi$ is uniformly continuous,
it is Lagrange stable. Consequently, without loss of generality, we may suppose that the
sequence $\varphi^{t_n}$ converges in the space $C(\mathbb R,X)$.
Thus the function $\psi$, figuring in the relations (\ref{eqAA1}),
belongs to $H(\varphi)$ and $H(\psi)\subseteq H(\varphi)$. It's
clear that the function $\psi$ is also Lagrange stable and,
consequently, without loss of generality we may suppose that the
sequence $\{\varphi^{t_{n_k}}\}$ converges to $\psi$ and
$\{\psi^{-t_{n_k}}\}$ converges to $\varphi$ in the space $C(\mathbb
R,X)$. The last fact contradicts to relations (\ref{eqAA2}) and
(\ref{eqAA3}). This contradiction proves our statement.
\end{proof}

\begin{remark} \rm
The function $\varphi(t)=\sin(\frac{1}{2+\cos t +\cos \sqrt{2}t})$ is
\begin{enumerate}
\item almost automorphic in the sense of Bochner \cite[Example 3.1]{BG};

\item Levitan almost periodic,
but it is not Bohr almost automorphic, because $\varphi$ is not
uniformly continuous on $\mathbb R$ \cite[Ch.V,
pp.212--213]{Lev_1953}.
\end{enumerate}
\end{remark}

\begin{definition} \rm
A function $\varphi \in C(\mathbb R,X)$ is called {\em quasi-periodic with the spectrum of frequencies $\nu_1,\nu_2,\ldots,\nu_k$} if
the following conditions are fulfilled:
\begin{enumerate}
\item the numbers $\nu_1,\nu_2,\ldots,\nu_k$ are rationally independent;
\item there exists a continuous function $\Phi :\mathbb R^{k}\to X$ such that
$\Phi(t_1+2\pi,t_2+2\pi,\ldots,t_k+2\pi)=\Phi(t_1,t_2,\ldots,t_k)$ for all $(t_1,t_2,\ldots,t_k)\in \mathbb R^{k}$;
\item $\varphi(t)=\Phi(\nu_1 t,\nu_2 t,\ldots,\nu_k t)$ for $t\in \mathbb R$.
\end{enumerate}
\end{definition}

Let $\varphi \in C(\mathbb R,X)$. Denote by $\mathfrak N_{\varphi}$ (respectively,
$\mathfrak M_{\varphi}$) the family of all sequences
$\{t_n\}\subset \mathbb R$ such that $\varphi^{t_n} \to \varphi$ (respectively,
$\{\varphi^{t_n}\}$ converges) in $C(\mathbb R,X)$ as $n\to \infty$.

By $\mathfrak N_{\varphi}^{u}$ (respectively, $\mathfrak M_{\varphi}^{u}$) we denote the family of sequences
$\{t_n\}\in \mathfrak N_{\varphi}$ such that $\varphi^{t_n}$ converges to $\varphi$
(respectively,  $\varphi^{t_n}$ converges) uniformly in $t\in\mathbb R$ as $n\to \infty$.

\begin{remark}\rm
\begin{enumerate}
\item The function $\varphi \in C(\mathbb R,X)$ is pseudo-periodic in the positive
(respectively, negative) direction if and only if there is a
sequence $\{t_n\}\in \mathfrak N_{\varphi}^{u}$ such that $t_n\to
+\infty$ (respectively, $t_n\to -\infty$) as $n\to \infty$.

\item Let $\varphi \in C(\mathbb R,X)$, $\psi \in C(\mathbb R,Y)$ and
$\mathfrak N_{\psi}^{u} \subseteq \mathfrak N_{\varphi}^{u}$. If the
function $\psi$ is pseudo-periodic in the positive (respectively,
negative) direction, then so is $\varphi$.
\end{enumerate}
\end{remark}

\begin{definition}\label{defPR}\rm 
A function $\varphi \in C(\mathbb
R,X)$ is called  {\em pseudo-recurrent} if for any
$\varepsilon >0$ and $l\in\mathbb R$ there exists $L\ge l$ such that for any $\tau_0\in \mathbb R$ we can find a number $\tau \in [l,L]$
satisfying
$$
\sup\limits_{|t|\le 1/\varepsilon}\rho(\varphi(t+\tau_0
+\tau),\varphi(t+\tau_0))\le \varepsilon.
$$
\end{definition}

\begin{remark}\label{remPR} \rm (\cite{Sch68,Sch72,Sch85,sib})
\begin{enumerate}
\item Every Birkhoff recurrent function is pseudo-recurrent, but the inverse
statement is not true in general.

\item If the function $\varphi \in C(\mathbb R,X)$ is
pseudo-recurrent, then every function $\psi\in H(\varphi)$ is
pseudo-recurrent.

\item If  the function $\varphi \in C(\mathbb R,X)$  is
Lagrange stable and  every function $\psi\in H(\varphi)$ is Poisson
stable, then $\varphi$ is pseudo-recurrent.
\end{enumerate}
\end{remark}

Finally, we remark that a Lagrange stable function is not Poisson stable in general, but
all other types of functions introduced above are Poisson stable.

\subsection{Shcherbakov's comparability method by character of recurrence}

\begin{definition} \rm
A function $\varphi \in C(\mathbb R,X)$ is said to be {\em comparable} (respectively, {\em uniformly comparable})
{\em by character of recurrence} with $\psi\in C(\mathbb R,Y)$ if $\mathfrak N_{\psi}\subseteq \mathfrak N_{\varphi}$
(respectively, $\mathfrak M_{\psi}\subseteq \mathfrak M_{\varphi}$).
\end{definition}

\begin{theorem}\label{th1}{\rm(\cite[ChII]{Sch72}, \cite{scher75})}
The following statements hold:
\begin{enumerate}
\item $\mathfrak M_{\psi}\subseteq \mathfrak M_{\varphi}$ implies $\mathfrak N_{\psi}\subseteq \mathfrak N_{\varphi}$, and hence
      uniform comparability implies comparability.

\item  Let $\varphi \in C(\mathbb R,X)$ be comparable by character of recurrence with $\psi\in C(\mathbb R,Y)$.
 If the function $\psi$ is stationary (respectively, $\tau$-periodic, Levitan almost periodic,
 almost recurrent, Poisson stable), then so is $\varphi$.
 \item Let $\varphi \in C(\mathbb R,X)$ be uniformly comparable by character of recurrence with $\psi\in C(\mathbb R,Y)$.
If the function $\psi$ is quasi-periodic with the spectrum of frequencies
$\nu_1,\nu_2,\dots,\nu_k$ (respectively, Bohr almost periodic,
Bohr almost automorphic, Birkhoff recurrent, Lagrange stable), then so is $\varphi$.

\item Let $\varphi \in C(\mathbb R,X)$ be uniformly comparable by character of recurrence with $\psi\in C(\mathbb R,Y)$
and $\psi$ be Lagrange stable. If $\psi$ is pseudo-periodic (respectively, pseudo-recurrent), then so is $\varphi$.
\end{enumerate}
\end{theorem}

\begin{lemma}\label{lAP1}
Let $\varphi \in C(\mathbb R,X)$, $\psi \in C(\mathbb R,Y)$ and $\mathfrak M_{\psi}^{u} \subseteq \mathfrak M_{\varphi}^{u}$.
Then the following statements hold:
\begin{enumerate}
\item  $\mathfrak N_{\psi}^{u} \subseteq \mathfrak N_{\varphi}^{u}$;
\item
If the function $\psi$ is Bohr almost periodic, then so is $\varphi$.
\end{enumerate}
\end{lemma}

\begin{proof}
Let $\mathfrak M_{\psi}^{u} \subseteq \mathfrak M_{\varphi}^{u}$ and $\{t_n\}\in \mathfrak N_{\psi}^{u}$.
Consider the sequence $\{\tilde{t}_n\}$ defined as follow: $\tilde{t}_{2k-1}=t_k$ and $\tilde{t}_{2k}=0$ for
 any $k\in\mathbb N$. It is clear that $\{\tilde{t}_n\}\in \mathfrak N_{\psi}^{u}
 \subseteq \mathfrak M_{\psi}^{u} \subseteq \mathfrak M_{\varphi}^{u}$; consequently, there exists
 a function $\tilde\varphi \in C(\mathbb R,X)$ such that the sequence $\{\varphi^{\tilde{t}_n}\}$
 converges to $\tilde{\varphi}$ uniformly on $\mathbb R$.
 Since $\{\varphi^{t_n}\}$ is a subsequence of $\{\varphi^{\tilde{t}_n}\}$, we have $\varphi = \tilde{\varphi}$,
 i.e. $\{t_n\}\in \mathfrak N_{\varphi}^{u}$.

Let $\psi \in C(\mathbb R,Y)$ be Bohr almost periodic, then
$\mathfrak M_{\psi}^{u}= \mathfrak M_{\psi}$ (see, for example, \cite[ChII]{Sch72},\cite{scher75}). Consequently, we have
$\mathfrak M_{\psi}=\mathfrak M_{\psi}^{u}\subseteq \mathfrak M_{\varphi}^{u}\subseteq \mathfrak M_{\varphi}$. This means that
the function $\varphi$ is uniformly comparable by character of recurrence with $\psi$ and according to Theorem \ref{th1} the function
$\varphi$ is Bohr almost periodic.
\end{proof}

\subsection{The function space $BUC$}

\begin{definition}\label{defQ1} \rm
A function $F:\mathbb R\times X\to X$ is called {\em continuous at $t_0\in\mathbb R$ uniformly with respect to (w.r.t.) $x\in
Q$} if for any $\varepsilon >0$ there exists $\delta
=\delta(t_0,\varepsilon)>0$ such that $|t-t_0|<\delta$ implies
$\sup\limits_{x\in Q}\rho(F(t,x),F(t_0,x))<\varepsilon$. The
function $F:\mathbb R\times X\to X$ is called {\em continuous on
$\mathbb R$ uniformly w.r.t. $x\in Q$} if it is continuous at every
point $t_0\in\mathbb R$ uniformly w.r.t. $x\in Q$.
\end{definition}

\begin{remark}\label{remQ1}\rm
If $Q$ is a compact subset of $X$ and $F:\mathbb R\times X\to
X$ is a continous function, then $F$ is continuous on $\mathbb R$
uniformly w.r.t. $x\in Q$.
\end{remark}

Denote by $BUC(\mathbb R\times X,X)$ the set of all functions $F:\mathbb R\times X\to X$ possessing the
following properties:
\begin{enumerate}
\item continuous in $t$ uniformly w.r.t. $x$ on every bounded subset
$Q\subseteq X$;

\item bounded on every bounded subset from $\mathbb R\times X$.
\end{enumerate}

For $F,G\in BUC(\mathbb R\times X,X)$ and $\{Q_n\}$ a sequence of
bounded subsets from $X$ such that $Q_n\subset Q_{n+1}$ for any
$n\in\mathbb N$ and $X=\bigcup_{n\ge 1} Q_n$,
denote
\begin{equation}\label{eqQD1}
d(F,G):=\sum_{n=1}^{\infty}\frac{1}{2^n}\frac{d_n(F,G)}{1+d_{n}(F,G)},
\end{equation}
where $d_{n}(F,G):=\sup\limits_{|t|\le n,\ x\in Q_n}\rho(F(t,x),G(t,x))$.
Then it is immediate to check that $d$ makes $BUC(\mathbb R\times X,X)$ a complete metric space and $d(F_k,F)\to
0$ if and only if $F_n(t,x)\to F(t,x)$ uniformly w.r.t. $(t,x)$ on
every bounded subset from $\mathbb R\times X$.

For given $F\in BUC(\mathbb R\times X,X)$ and $\tau\in\mathbb R$,
denote by $F^\tau$ the {\em translation of $F$}, i.e. $F^\tau(t,x):=
F(t+\tau,x)$ for $(t,x)\in \mathbb R\times X$, and the {\em hull of
$F$} by $H(F):=\overline{\{F^\tau:\tau\in\mathbb R\}}$ with the
closure being taken under the metric $d$ given by \eqref{eqQD1}. It is immediate to check
that the mapping $\sigma: \mathbb R\times BUC(\mathbb R\times X,X)
\to BUC(\mathbb R\times X,X)$ defined by $\sigma(\tau,F):= F^\tau$
is a dynamical system, i.e. $\sigma(0,F) = F$,
$\sigma(\tau_1+\tau_2,F) = \sigma(\tau_2,\sigma(\tau_1,F))$ and the
mapping $\sigma$ is continuous. See \cite[\S 1.1]{Ch2015} for
details.

Denote by  $BC(X,X)$ the set of all
continuous and bounded on every bounded subset $Q\subset X$
functions $F:X\to X$ and let
$$
d_{BC}(F,G):=\sum_{n=1}^{\infty}\frac{1}{2^n}\frac{d_n(F,G)}{1+d_{n}(F,G)}
$$
for any $F,G\in BC(X,X)$, where $d_{n}(F,G):=\sup\limits_{x\in Q_n}\rho(F(x),G(x))$.
Then $BC(X,X)$ is a complete metric space.

Let now  $F\in BUC(\mathbb R\times X,X)$ and $\mathcal F :\mathbb R\to BC(X,X)$ a
mapping defined by equality $\mathcal F(t):=F(t,\cdot)$.

\begin{remark}\label{remBUC} \rm
It is not difficult to check that:
\begin{enumerate}
\item  $\mathfrak M_{F}=\mathfrak
M_{\mathcal F}$ for any $F\in BUC(\mathbb R\times X,X)$;

\item $\mathfrak M_{F}^{u}=\mathfrak
M_{\mathcal F}^{u}$ for any $F\in BUC(\mathbb R\times X,X)$.
\end{enumerate}
Here $\mathfrak M_{F}$ is the set of all sequences $\{t_n\}$ such that
$F^{t_n}$ converges in the space $BUC(\mathbb R\times X,X)$ and $\mathfrak M_{F}^{u}$ is
the set of all sequences $\{t_n\}$ such that $F^{t_n+t}$ converges
in the space $BUC(\mathbb R\times X,X)$ uniformly w.r.t. $t\in\mathbb R$.
\end{remark}

\section{Linear equations}

Let $\mathfrak B$ be a Banach space with the norm
$|\cdot|_{\mathfrak B}$. Consider the linear nonhomogeneous equation
\begin{equation}\label{eqLN1}
\dot{x}=Ax+f(t)
\end{equation}
on the space $\mathfrak B$, where $f\in C(\mathbb R,\mathfrak B)$ and $A$ is an infinitesimal generator which generates a
$C_0$-semigroup $\{U(t)\}_{t\ge 0}$ acting on $\mathfrak B$.


\begin{definition} \rm
A semigroup of operators $\{U(t)\}_{t\ge 0}$ is said to be
{\em  exponentially stable}, if there are positive
numbers $\mathcal N,\nu >0$ such that $||U(t)||\le \mathcal N
e^{-\nu t}$ for any $t\ge 0$.
\end{definition}

Denote by $C_{b}(\mathbb R,\mathfrak B)$ the Banach space of all
continuous and bounded mappings $\varphi :\mathbb R\to \mathfrak
B$ equipped with the norm
$||\varphi||_{\infty}:=\sup\{|\varphi(t)|_{\mathfrak B}:\
t\in\mathbb R\}$.



Let $(H,|\cdot|)$ be a real separable Hilbert space, $(\Omega ,
\mathcal F,\mathbb P)$ be a probability space, and $L^{2}(\mathbb
P,H)$ be the space of $H$-valued random variables $x$ such that
\begin{equation*}\label{eqLSDE03}
\mathbb E|x|^2 :=\int\limits_{\Omega}|x|^2 d\mathbb P<\infty .
\end{equation*}
Then $L^2(\mathbb P,H)$ is a Hilbert space equipped with the norm
\begin{equation*}\label{eqLSDE04}
||x||_2:=\Big{(}\int\limits_{\Omega}|x|^2  d\mathbb P\Big{)}^{1/2}.
\end{equation*}
For $f\in C_b(\mathbb R, L^2(\mathbb P,H))$, the space of bounded continuous mappings from $\mathbb R$ to $L^2(\mathbb P,H)$, we denote
$||f||_{\infty}:=\sup\limits_{t\in\mathbb R}||f(t)||_2$.

Consider the following semi-linear stochastic differential equation
\begin{equation}\label{SlSDE}
dx(t)=(Ax(t)+f(t,x(t)))dt+g(t,x(t))dW(t), 
\end{equation}
where $A$ is an infinitesimal generator which generates a
$C_0$-semigroup $\{U(t)\}_{t\ge 0}$, $f,g: \mathbb R\times H\to H$ and $W(t)$ is a two-sided
standard one-dimensional Brownian motion defined on the probability space
$(\Omega,\mathcal F,\mathbb P)$.  We set $\mathcal F_{t}:=\sigma\{W(u):  u \le t\}$.

\begin{definition} \rm
Recall that an $\mathcal F_{t}$-adapted processes $\{x(t)\}_{t\in\mathbb R}$ is said to be
a mild solution of equation (\ref{SlSDE}) if it satisfies the stochastic integral equation
$$
x(t)=U(t-t_0)x(t_0)+\int_{t_0}^{t}U(t-s)f(s,x(s))ds+\int_{t_0}^{t}U(t-s)g(s,x(s))dW(s) ,
$$
for all $t\ge t_0$ and each $t_0\in \mathbb R$.
\end{definition}

\begin{remark}\label{remL1} \rm
If $\varphi \in C_{b}(\mathbb R,L^2(\mathbb
P,H))$, then for any $\psi \in H(\varphi)$ we have $||\psi(t)||_2 \le
||\varphi||_{\infty}$ for every $t\in\mathbb R$.
\end{remark}

Let $\mathcal P(H)$ be the space of all Borel probability measures on $H$ endowed
with the $\beta$ metric:
$$
\beta (\mu,\nu) :=\sup\left\{ \left| \int f d \mu - \int fd \nu\right|: ||f||_{BL} \le 1
\right\}, \quad \hbox{for }\mu,\nu\in \mathcal P(H),
$$
where $f$ are bounded Lipschitz continuous real-valued functions on $H$ with the norms
\[
||f||_{BL}= Lip(f) + ||f||_\infty,~ Lip(f)=\sup_{x\neq y}
\frac{|f(x)-f(y)|}{|x-y|},~ ||f||_{\infty}=\sup_{x\in H}|f(x)|.
\]
Recall that a sequence $\{\mu_n\}\subset \mathcal P(H)$ is
said to {\em weakly converge} to $\mu$ if $\int f d\mu_n\to \int f d\mu$
for all $f\in C_b(H)$, where $C_b(H)$ is the space
of all bounded continuous real-valued functions on $H$. It is well-known that
$(\mathcal P(H),\beta)$ is a separable complete metric space and that a sequence $\{\mu_n\}$
weakly converges to $\mu$ if and only if $\beta(\mu_n, \mu) \to0$ as $n\to\infty$.

\begin{definition}\label{aad}\rm
A sequence of random variables $\{x_n\}$ is said to \emph{converge in distribution} to the random variable $x$ if the corresponding laws
$\{\mu_n\}$ of $\{x_n\}$ weakly converge to the law $\mu$ of $x$, i.e. $\beta(\mu_n,\mu)\to 0$.
\end{definition}

\begin{definition} \rm
Let $\{\varphi (t)\}_{t\in\mathbb R}$ be a mild solution of equation \eqref{SlSDE}. Then $\varphi$ is called
{\em comparable} (respectively, {\em uniformly comparable}) {\em in distribution}
if $\mathfrak N_{(f,g)} \subseteq \tilde{\mathfrak N}_{\varphi}$ (respectively, $\mathfrak M_{(f,g)} \subseteq \tilde{\mathfrak M}_{\varphi}$),
where $\tilde{\mathfrak N}_{\varphi}$ (respectively, $\tilde{\mathfrak M}_{\varphi}$) means the set of all sequences $\{t_n\}\subset\mathbb R$
such that the sequence $\{\varphi(\cdot+t_n)\}$ converges to $\varphi(\cdot)$
(respectively, $\{\varphi(\cdot+t_n)\}$ converges) in distribution uniformly on any compact interval.
\end{definition}

In this section, we consider the following linear stochastic differential equation
\begin{equation}\label{eqLSDE1}
dx(t)=(Ax(t)+f(t))dt+g(t)dW(t), 
\end{equation}
where $A$ and $W$ are the same as in \eqref{SlSDE}, and $f,g\in C(\mathbb R,L^2(\mathbb P,H))$ are $\mathcal F_t$-adapted.

\begin{theorem}\label{thLS11_2}
Consider the equation \eqref{eqLSDE1}.
Suppose that the semigroup $\{U(t)\}_{t\ge 0}$ acting on $H$
is exponentially stable, then the following statements hold:
\begin{enumerate}
\item for every $f,g\in C_{b}(\mathbb R,L^2(\mathbb P,H))$
there exists a unique solution $\varphi \in C_{b}(\mathbb R,$ $L^2(\mathbb P,H))$ of equation
\eqref{eqLSDE1}
given by the formula
\begin{equation}\label{eqD3}
\varphi(t) =\int_{-\infty}^{t}U(t-\tau) f(\tau)d\tau +
 \int_{-\infty}^{t}U(t-\tau) g(\tau)dW(\tau);
\end{equation}

\item the Green's operator $\mathbb G$ defined by
\begin{equation*}\label{eqG10}
\mathbb G(f,g)(t):=\int_{-\infty}^{t}  U(t-\tau)f(\tau)d\tau + \int_{-\infty}^{t}U(t-\tau)g(\tau)dW(\tau)
\end{equation*}
is a bounded operator defined on $C_b(\mathbb R,L^2(\mathbb P,H))\times C_b(\mathbb R,L^2(\mathbb P,H))$
with values in $C_b(\mathbb R,L^2(\mathbb P,H))$
and
\begin{equation}\label{Ge}
||\mathbb G(f,g)||_{\infty}\le \frac{\mathcal N}{\nu}
\Big{(}2||f||^2_{\infty}+\nu ||g||^2_{\infty}\Big{)}^{1/2}
\end{equation}
for all $f,g\in C_b(\mathbb R, L^2(\mathbb P,H))$;

\item if $f,g\in C_{b}(\mathbb R,L^2(\mathbb P,H))$ and $l>L>0$, then
\begin{align}\label{eqL1}
\max\limits_{|t|\le L}\mathbb E|\varphi(t)|^2 & \le
 \frac{\mathcal N^2}{\nu^2}\big{(}2\max\limits_{|\tau|\le l}\mathbb E|f(\tau)|^2 +\nu\max\limits_{|\tau|\le
l}\mathbb E|g(\tau)|^2\big{)} \nonumber \\
&\qquad +
\frac{\mathcal N^2}{\nu^{2}} \big{(} 2 e^{-\nu(l-L)} ||f||^{2}_{\infty}+
\nu e^{-2\nu(l-L)} ||g||^{2}_{\infty}\big{)};
\end{align}

\item if $f,g\in C_{b}(\mathbb R,L^2(\mathbb P,H))$, then the unique
$L^2$-bounded solution $\varphi$ of equation \eqref{eqLSDE1} is uniformly comparable in distribution;

\item  $\mathfrak M_{(f,g)}^{u}\subseteq \mathfrak {\tilde{M}}_{\varphi}^{u}$, where ${\tilde{\mathfrak M}}_{\varphi}^{u}$ is the
set of all sequences $\{t_n\}$ such that the sequence $\{\varphi(t+t_n)\}$ converges in distribution uniformly in $t\in\mathbb R$.
\end{enumerate}
\end{theorem}

\begin{proof}
(i)-(ii). It is straightforward to verify that the function $\varphi$
given by \eqref{eqD3} is a solution of the equation \eqref{eqLSDE1}.
If $\psi\in C_{b}(\mathbb R,$ $L^2(\mathbb P,H))$ is also a
solution, then $u(t)=\varphi(t)-\psi(t)$ satisfies the equation
\[
x'(t)=Ax(t).
\]
But under the exponential stability condition of $\{U(t)\}_{t\ge
0}$, this equation has only trivial solution in $C_{b}(\mathbb R,$
$L^2(\mathbb P,H))$. This enforces that $\varphi=\psi$.

We now show the boundedness of $\varphi$. Note that $\varphi(t)=p(t)+q(t)$ for $t\in\mathbb R$, where
\begin{equation*}\label{eqD4}
p(t):= \int
_{-\infty}^{t}U(t-\tau)f(\tau)d\tau
\end{equation*}
and
\begin{equation*}
 q(t):= \int
_{-\infty}^{t}U(t-\tau)g(\tau)dW(\tau).
\end{equation*}
For the first term, by the Cauchy-Schwarz inequality we have
\begin{equation}
\label{eqB19}
\begin{split}
\mathbb E|p(t)|^2 & =\mathbb E \left|\int_{-\infty}^{t} U(t-\tau)f(\tau)d\tau \right|^2 \\
&\le \mathbb E\left[\int_{-\infty}^{t}||U(t-\tau)||\cdot|f(\tau)|d\tau\right]^2
\le \mathbb E\left[\int_{-\infty}^{t} \mathcal N e^{-\nu (t-\tau)}
|f(\tau)|d\tau\right]^2    \\
& \le \mathcal N^2 \int_{-\infty}^{t} e^{-\nu (t-\tau)}d\tau
\int_{-\infty}^{t} e^{-\nu (t-\tau)}\mathbb E |f(\tau)|^2 d\tau   \\
&  \le  \frac{\mathcal N^2}{\nu^2}||f||_{\infty}^2.
\end{split}
\end{equation}
For the second term, using It\^o's isometry property we get
\begin{align}\label{eqB29}
\mathbb E|q(t)|^2 &= \mathbb E\left|\int
_{-\infty}^{t}U(t-\tau)g(\tau)d W(\tau) \right|^2 = \int
_{-\infty}^{t}\mathbb E |U(t-\tau)g(\tau)|^2d\tau \nonumber \\
&\le \int _{-\infty}^{t}\mathcal N^2e^{-2\nu (t-\tau)}\mathbb E
|g(\tau)|^2d\tau \le ||g||_{\infty}^2\int
_{-\infty}^{t}\mathcal N^2e^{-2\nu (t-\tau)}d\tau\nonumber\\
 &=
\frac{\mathcal N^2}{2\nu}||g||_{\infty}^2.
\end{align}
From (\ref{eqB19}) and (\ref{eqB29}) we have
\begin{align*}\label{eqB3}
 \mathbb E|\varphi(t)|^2 \le 2(\mathbb
E|p(t)|^2+\mathbb E|q(t)|^2) \le
 \frac{\mathcal
N^2}{\nu^{2}}\big{(}2||f||^2_{\infty}+\nu ||g||^2_{\infty}\big{)},
\end{align*}
and consequently
\begin{equation*}\label{eqB4}
||\mathbb G(f,g)||_{\infty}= ||\varphi||_{\infty}\le \frac{\mathcal N}{\nu}
\Big{(}2||f||^2_{\infty}+\nu ||g||^2_{\infty}\Big{)}^{1/2}.
\end{equation*}

(iii).  Let $L>0$, $t\in
[-L,L]$, $l>L$ and $f,g\in C_{b}(\mathbb R,L^{2}(\mathbb P,H))$,
then from (\ref{eqB19}), (\ref{eqB29}) we have
\begin{align}\label{eqBL1}
 \mathbb E|\varphi(t)|^2 & \le 2(\mathbb E|p(t)|^2+\mathbb
E|q(t)|^2) \nonumber\\
& \le 2\mathcal N^2
\left(\frac{1}{\nu}\int_{-\infty}^{t}e^{-\nu(t-\tau)}\mathbb
E|f(\tau)|^2d\tau + \int_{-\infty}^{t}e^{-2\nu(t-\tau)}\mathbb
E|g(\tau)|^2d\tau\right).
\end{align}
Note that
\begin{align}\label{eqBL2}
\int_{-\infty}^{t}e^{-\nu(t-\tau)}\mathbb E|f(\tau)|^2d\tau
& = \int_{-l}^t e^{-\nu(t-\tau)}\mathbb
E|f(\tau)|^2d\tau + \int_{-\infty}^{-l} e^{-\nu(t-\tau)}\mathbb
E|f(\tau)|^2d\tau \nonumber \\
&   \le \frac{1}{\nu}\max\limits_{|\tau|\le
l}\mathbb E|f(\tau)|^2 + \frac{1}{\nu}e^{-\nu (t+l)}||f||^{2}_{\infty}.
\end{align}
Reasoning as above we have
\begin{equation}\label{eqBL4}
 \int_{-\infty}^{t}e^{-2\nu(t-\tau)}\mathbb E|g(\tau)|^2d\tau
\le \frac{1}{2\nu}\max\limits_{|\tau|\le l}\mathbb E|g(\tau)|^2 +
\frac{1}{2\nu}e^{-2\nu (t+l)}||g||^{2}_{\infty} .
\end{equation}
From (\ref{eqBL1})-(\ref{eqBL4}) we obtain
\begin{align*}
\max\limits_{|t|\le L}\mathbb E|\varphi(t)|^2 & \le
 \frac{\mathcal N^2}{\nu^{2}}\big{(}2\max\limits_{|\tau|\le l}\mathbb E|f(\tau)|^2 +\nu \max\limits_{|\tau|\le
l}\mathbb E|g(\tau)|^2\big{)} \\
&\qquad + \frac{\mathcal N^2}{\nu^{2}} \big{(} 2 e^{-\nu(l-L)} ||f||^{2}_{\infty}+
\nu e^{-2\nu(l-L)} ||g||^{2}_{\infty}\big{)}.
\end{align*}
Thus inequality (\ref{eqL1}) is established.

(iv). Let now $\{t_n\}\in \mathfrak M_{(f,g)}$, then there exists
$(\tilde{f},\tilde{g})\in H(f,g)$ such that $f^{t_n}\to \tilde{f}$
and $g^{t_n}\to \tilde{g}$ in the space
$C(\mathbb R, L^{2}(\mathbb P,H))$ as $n\to \infty$; that is, for any $L>0$ we have
\begin{equation*}\label{eqD6}
\max\limits_{|t|\le L}\mathbb E|f(t+t_n)-\tilde{f}(t)|^2\to 0\ \
\mbox{and}\ \ \max\limits_{|t|\le L}\mathbb
E|g(t+t_n)-\tilde{g}(t)|^2\to 0
\end{equation*}
as $n\to \infty$.

Denote by $h^{1}_{n}(t):=f^{t_n}(t)-\tilde{f}(t)$ and
$h^{2}_{n}(t):=g^{t_n}(t)-\tilde{g}(t)$ for any $t\in\mathbb R$,
$\varphi_{n}:=\mathbb G(f^{t_n},g^{t_n})$, $\tilde{\varphi}:=\mathbb
G(\tilde{f},\tilde{g})$ and $\psi_{n}:=\varphi_{n}-\tilde{\varphi}$.
It is easy to check that $\psi_{n}=\mathbb G(h^{1}_{n},h^{2}_{n})$,
$h^{i}_{n}\in C_{b}(\mathbb R,L^{2}(\mathbb R,H))$ ($i=1,2$) and by
Remark \ref{remL1} we have $||h_{n}^{1}||_{\infty}\le
2||f||_{\infty}$ (respectively, $||h_{n}^{2}||_{\infty}\le
2||g||_{\infty}$). Let now $\{l_n\}$ be a sequence of positive
numbers such that $l_n\to \infty$ as $n\to \infty$. According to
inequality (\ref{eqL1}) we obtain
\begin{align}\label{eqBL6}
\max\limits_{|t|\le L}\mathbb E|\psi_{n}(t)|^2 & \le
 \frac{\mathcal N^2}{\nu^2}\big{(}2\max\limits_{|\tau|\le l_n}\mathbb E|h_{n}^{1}(\tau)|^2 +\nu \max\limits_{|\tau|\le
l_n}\mathbb E|h_n^{2}(\tau)|^2\big{)} \nonumber \\
& \qquad + \frac{\mathcal N^2}{\nu^2} \big{(}2 e^{-\nu(l_n-L)} ||h_n^1||^{2}_{\infty}+
\nu e^{-2\nu(l_n-L)} ||h_n^2||^{2}_{\infty}\big{)}.
\end{align}
Passing to limit in (\ref{eqBL6}) as $n\to \infty$ and taking into
consideration Remark \ref{remD1}-(iii) we have
\begin{equation*}\label{eqBL8}
\lim\limits_{n\to \infty}\max\limits_{|t|\le L}\mathbb
E|\psi_{n}(t)|^2=0
\end{equation*}
for any $L>0$; that is, $\varphi_{n}\to \tilde{\varphi}$ in the space $C(\mathbb R,L^{2}(\mathbb P,H))$ as $n\to
\infty$.

Since $L^2$ convergence implies convergence in distribution, we have
$\varphi_n(t)\to \tilde{\varphi}(t)$ in distribution uniformly in
$t\in [-L,L]$ for all $L>0$. On the other hand we have
$$
\varphi(t+t_n)= \int _{-\infty}^{t} U (t-\tau)f(\tau +t_n)d\tau
+\int _{-\infty}^{t} U(t-\tau)g(\tau +t_n)d\tilde{W}_n(\tau),
$$
where $\tilde{W}_n(\tau):=W(\tau +t_n)-W(t_n)$ is a shifted Brownian motion.
So $\varphi_n(t)$ and $\varphi(t+t_n)$ share the same distribution on $H$, and hence
$\varphi(t+t_n)\to \tilde{\varphi}(t)$ in distribution uniformly in
$t\in [-L,L]$ for all $L>0$. Thus we have $\{t_n\}\in \mathfrak
{\tilde{M}}_{\varphi}$. That is, $\varphi$ is uniformly comparable in distribution.

(v). Let $\{t_n\}\in \mathfrak M_{(f,g)}^{u}$, then there
exists $(\tilde{f},\tilde{g})\in H(f,g)$ such that $f^{t_n}\to
\tilde{f}$ and $g^{t_n}\to \tilde{g}$ uniformly in
$t\in \mathbb R$ as $n\to \infty$, that is,
\begin{equation*}\label{eqPP1}
\max\limits_{t\in \mathbb R}\mathbb E|f(t+t_n)-\tilde{f}(t)|^2\to 0\
\  \mbox{and}\ \ \max\limits_{t\in \mathbb R}\mathbb
E|g(t+t_n)-\tilde{g}(t)|^2\to 0
\end{equation*}
as $n\to \infty$.
As above we denote by $h^{1}_{n}(t):=f^{t_n}(t)-\tilde{f}(t)$ and
$h^{2}_{n}(t):=g^{t_n}(t)-\tilde{g}(t)$ for $t\in\mathbb R$,
$\varphi_{n}:=\mathbb G(f^{t_n},g^{t_n})$, $\tilde{\varphi}:=\mathbb
G(\tilde{f},\tilde{g})$ and $\psi_{n}:=\varphi_{n}-\tilde{\varphi}$.

According to inequality (\ref{Ge}) we obtain
\begin{eqnarray}\label{eqBL6_1}
||\psi_{n}||_{\infty}\le
 \frac{\mathcal N }{\nu}\big{(}2||h_{n}^{1}||^2_{\infty} +\nu ||h_n^{2}||^2_{\infty}\big{)}^{1/2}.
\end{eqnarray}
Passing to limit in (\ref{eqBL6_1}) we obtain $\varphi_{n}\to
\tilde{\varphi}$ uniformly on $\mathbb R$ in $L^2$-norm as
$n\to\infty$. Since $\varphi_{n}(t)$ and $\varphi(t+t_n)$ have the
same distributions,  $\varphi(t+t_n)\to \tilde{\varphi}(t)$ in
distribution uniformly in $t\in \mathbb R$. Thus we have $\{t_n\}\in
\mathfrak {\tilde{M}}^{u}_{\varphi}$. The proof is complete.
\end{proof}

\begin{coro}\label{cor 10}
Under the conditions of \textsl{Theorem} \ref{thLS11_2} if the functions
$f,g\in$ $ C_{b}(\mathbb R,$ $L^2(\mathbb P,H))$ are jointly stationary (respectively,
$\tau$-periodic, quasi-periodic with the spectrum of frequencies $\nu_1,\ldots,\nu_k$,
Bohr almost periodic, Bohr almost automorphic, Birkhoff recurrent, Lagrange stable,
Levitan almost periodic, almost recurrent, Poisson stable), then equation (\ref{eqLSDE1})
has a unique solution $\varphi \in C_{b}(\mathbb R,L^2(\mathbb P,H))$ which is
stationary (respectively,
$\tau$-periodic, quasi-periodic with the spectrum of frequencies $\nu_1,\ldots,\nu_k$,
Bohr almost periodic, Bohr almost automorphic, Birkhoff recurrent, Lagrange stable,
Levitan almost periodic, almost recurrent, Poisson stable) in distribution. If the functions
$f,g\in$ $ C_{b}(\mathbb R,$ $L^2(\mathbb P,H))$ are jointly Lagrange stable and jointly pseudo-periodic (respectively,
pseudo-recurrent), then equation (\ref{eqLSDE1})
has a unique solution $\varphi \in C_{b}(\mathbb R,L^2(\mathbb P,H))$ which is
pseudo-periodic (respectively, pseudo-recurrent) in distribution.
\end{coro}

\begin{proof}
This statement follows from Theorems \ref{th1} and \ref{thLS11_2}.
\end{proof}

\section{Semi-linear equations}


Let us consider the stochastic differential equation
\begin{equation}\label{eq3.4.4}
dx(t)=(Ax(t)+F(t,x(t)))dt + G(t,x(t))dW(t),
\end{equation}
where $F,G\in C(\mathbb{R}\times  H, H)$.

\begin{definition}\label{defL1} \rm
We say that the functions $F$
and $G$ satisfy the condition
\begin{enumerate}

\item[(C1)] if there exists a number $A_0\ge 0$ such that $|F(t,0)|, |G(t,0|\le
A_0$ for any $t\in \mathbb R$;

\item[(C2)] if there exists a number $\mathcal L\ge 0$ such that $Lip(F), Lip(G)\le
\mathcal L$, where
$$
Lip(F):=\sup\left\{\frac{|F(t,x_1)-F(t,x_2)|}{|x_1-x_2|}:\ x_1\not= x_2,\
t\in\mathbb R\right\};
$$
\item[(C3)] if $F$ and $G$ are continuous in $t$ uniformly w.r.t. $x$ on each bounded subset $Q\subset H$.
\end{enumerate}
\end{definition}

\begin{remark}\label{remL10} \rm
\begin{enumerate}
\item If $F$ and $G$ satisfy (C1)-(C2) with the constants $A_0$
and $\mathcal L$, then every pair of functions
$(\tilde{F},\tilde{G})$ in $H(F,G):= \overline{\{(F^{\tau},G^{\tau}): \tau\in \mathbb R\}}$, the hull of $(F,G)$, also posses the same
property with the same constants.

\item If $F$ and $G$ satisfy the conditions
(C1)-(C3), then $F,G\in BUC (\mathbb R\times H, H)$ and
$H(F,G)\subset BUC (\mathbb R\times H, H) \times BUC (\mathbb
R\times H, H)$.

\item When we consider stochastic ordinary differential
equations, i.e. $F,G\in C(\mathbb{R}\times  \mathbb R^d, \mathbb
R^d)$, the condition (C3) naturally holds as pointed out in Remark
\ref{remQ1}; but for stochastic partial differential equations, we
need to check condition (C3) carefully.
\end{enumerate}
\end{remark}

\begin{lemma}\label{lB} Let $u,f\in C_{b}(\mathbb R,\mathbb
R_{+})$ and $\nu >\alpha \ge 0$, then the following statements hold:
\begin{enumerate}
\item if
\begin{equation*}\label{eqB1}
u(t)\le \int_{-\infty}^{t}e^{-\nu (t-\tau)}(\alpha u(\tau) +
f(\tau))d\tau
\end{equation*}
for any $t\in \mathbb R$, then
\begin{equation*}\label{eqB2}
u(t)\le \int_{-\infty}^{t}e^{-k(t-\tau)}f(\tau)d\tau,
\end{equation*}
where $k:=\nu -\alpha$;
\item if $l>L>0$, then
\begin{equation*}\label{eqB3}
\max\limits_{|t|\le L}u(t)\le
\frac{e^{kL}e^{-kl}}{k}\sup\limits_{t\in\mathbb
R}f(t)+\frac{1-e^{-kL}e^{-kl}}{k} \max\limits_{|t|\le l}f(t).
\end{equation*}
\end{enumerate}
\end{lemma}

\begin{proof}
(i). Consider the equation
\begin{equation}\label{eqB4}
v(t)= \int_{-\infty}^{t}e^{-\nu (t-\tau)}(\alpha v(\tau) +
f(\tau))d\tau.
\end{equation}
Note that the linear operator $ \mathcal A: C_{b}(\mathbb R,\mathbb
R)\to  C_{b}(\mathbb R,\mathbb R)$ defined by
$$
(\mathcal A\varphi)(t):= \int_{-\infty}^{t}e^{-\nu (t-\tau)}\alpha
\varphi(\tau)d\tau \
$$
is a contraction, where $C_{b}(\mathbb R,\mathbb R)$ is equipped
with the norm $||\varphi||_\infty:=\sup\{|\varphi(t)|:\ t\in\mathbb R\}$.
In fact, it is immediate to check that $||\mathcal A||\le
\frac{\alpha}{\nu}<1$, $C_{b}(\mathbb R,\mathbb R_{+})$ is a cone in
the space $C_{b}(\mathbb R,\mathbb R)$ and $\mathcal A
(C_{b}(\mathbb R,\mathbb R_{+}))\subseteq C_{b}(\mathbb R,\mathbb
R_{+})$. Thus the operator $\Phi : C_{b}(\mathbb R,\mathbb R)\to
C_{b}(\mathbb R,\mathbb R)$, defined by
\[
(\Phi
\phi)(t):=\int_{-\infty}^{t}e^{-\nu (t-\tau)}(\alpha \phi(\tau) +
f(\tau))d\tau,\quad\hbox{for } t\in\mathbb R
\]
is a contraction and consequently the equation (\ref{eqB4}) has a unique solution on the
space $C_{b}(\mathbb R,\mathbb R)$.

Note that the unique bounded solution $v(t)$ of equation (\ref{eqB4}) is a solution of the equation
\begin{equation*}\label{eqB5}
v'(t)=-kv(t)+f(t)
\end{equation*}
and consequently it is given by
\begin{equation*}\label{eqB6}
v(t)= \int_{-\infty}^{t}e^{-k(t-\tau)}f(\tau)d\tau.
\end{equation*}

Since $C_{b}(\mathbb R,\mathbb R_{+})$ is a cone in the space
$C_{b}(\mathbb R,\mathbb R)$ and $\mathcal A (C_{b}(\mathbb
R,\mathbb R_{+}))\subseteq C_{b}(\mathbb R,\mathbb R_{+})$, to
finish the proof of the first statement we note that by
\cite[ChI, Theorem 9.3]{Dal} we have $u(t)\le v(t)$ for all $t\in\mathbb R$.

(ii). Let now $l>L>0$ and $t\in [-L,L]$, then we have
\begin{align*}
\int_{-\infty}^{t}e^{-k(t-\tau)}f(\tau)d\tau &=
\int_{-\infty}^{-l}e^{-k(t-\tau)}f(\tau)d\tau
+\int_{-l}^{t}e^{-k(t-\tau)}f(\tau)d\tau  \nonumber \\
& \le \sup\limits_{t\in\mathbb R}f(t)\, \cdot
\,\frac{e^{-k(t+l)}}{k} +\max\limits_{|t|\le
l}f(t)\cdot\frac{1-e^{-k(t+l)}}{k}.
\end{align*}
Consequently,
\begin{equation*}\label{eqB8}
\max\limits_{|t|\le L}u(t)\le \max\limits_{|t|\le L}v(t) \le
\frac{e^{kL}e^{-kl}}{k}\sup\limits_{t\in\mathbb
R}f(t)+\frac{1-e^{-kL}e^{-kl}}{k} \max\limits_{|t|\le l}f(t).
\end{equation*}
\end{proof}

\begin{prop}\label{Lp}
Consider the equation \eqref{eq3.4.4}. Suppose that the following conditions hold:
\begin{enumerate}
\item the semigroup $\{U(t)\}_{t\ge 0}$ acting on the space $H$
is exponentially stable;
\item $F,G\in C(\mathbb R\times  H,  H)$;
\item the functions $F$ and $G$ satisfy the conditions (C1) and
(C2).
\end{enumerate}
For $p> 2$, denote
\[
c_p := \left[ \frac{p(p-1)}{2}\cdot \left(\frac{p}{p-1}\right)^{p-2}  \right]^{p/2}.
\]
If
\[
\theta_p:= 2^{p-1}\mathcal N^p \mathcal L^p \left[ \left(\frac{2(p-1)}{\nu p}\right)^{p-1}
+ c_p  \left(\frac{p-2}{\nu p} \right)^{p/2-1} \right]
\cdot\frac{2}{\nu p} <1,
\]
then \eqref{eq3.4.4} admits a unique bounded solution in $C_b(\mathbb R,L^p(\mathbb P,H))$.
\end{prop}

\begin{proof}
Since the semigroup $U(t)$ is exponentially stable, it can be checked that $x_0\in C_b(\mathbb R,L^p(\mathbb P,H))$
is a mild solution of \eqref{eq3.4.4} if and only if it satisfies the integral equation
\[
x_0(t)=\int_{-\infty}^{t} U(t-\tau) F(\tau, x_0(\tau)) d\tau
+  \int_{-\infty}^{t} U(t-\tau) G(\tau, x_0(\tau)) dW(\tau).
\]
Define an operator $\mathcal S$ on $C_b(\mathbb R,L^p(\mathbb P,H))$ by
\begin{align*}
(\mathcal S x) (t) := \int_{-\infty}^{t} U(t-\tau) F(\tau, x(\tau)) d\tau
+  \int_{-\infty}^{t} U(t-\tau) G(\tau, x(\tau)) dW(\tau).
\end{align*}
Since $F$, $G$ satisfy the conditions (C1) and (C2), it is not hard to check that $\mathcal S$ maps $C_b(\mathbb R,L^p(\mathbb P,H))$
into itself.

By the proof of \cite[Theorem 4.36]{DZ}, we have for any $s<t$
\[
\mathbb E \left|\int_s^t f(\tau) dW(\tau) \right|^p \le c_p \left( \mathbb E\int_s^t |f(\tau)|^2  d\tau \right)^{p/2}.
\]
So by H\"older's inequality with exponents $(p,\frac{p}{p-1})$ and $(\frac{p}{2},\frac{p}{p-2})$ respectively,
we have  for $x,y\in C_b(\mathbb R,L^p(\mathbb P,H))$ and $t\in\mathbb R$
\begin{align*}
&\mathbb E|(\mathcal S x)(t) - (\mathcal S y)(t)|^p \\
& \le 2^{p-1} \bigg[ \mathbb E \left|\int_{-\infty}^{t} U(t-\tau)( F(\tau, x(\tau)) - F (\tau, y(\tau)) ) d\tau \right|^p \\
&\quad + \mathbb E \left| \int_{-\infty}^{t}U(t-\tau) (G(\tau, x(\tau)) -G(\tau, y(\tau)) )dW(\tau)\right|^p\bigg]\\
& \le 2^{p-1} \bigg[ \mathbb E\left(\int_{-\infty}^{t}  \mathcal N e^{-\nu (t-\tau)} \mathcal L |x(\tau)-y(\tau)| d\tau\right)^p \\
&\quad + c_p \left( \mathbb E \int_{-\infty}^{t}\mathcal N^2 e^{-2\nu(t-\tau)} \mathcal L^2 |x(\tau)-y(\tau)|^2 d\tau \right)^{p/2} \bigg]\\
& \le 2^{p-1}\mathcal N^p \mathcal L^p \bigg[ \left(\int_{-\infty}^{t} e^{-\frac12\nu \frac{p}{p-1} (t-\tau)} d\tau \right)^{p-1} \cdot
\mathbb E \int_{-\infty}^{t} e^{-\frac12\nu p (t-\tau)} |x(\tau)-y(\tau)|^p d\tau\\
& \quad + c_p \left( \int_{-\infty}^{t} e^{- \frac{\nu p}{p-2} (t-\tau)} d\tau \right)^{p/2-1}
\cdot \mathbb E \int_{-\infty}^{t} e^{-\frac12\nu p (t-\tau)} |x(\tau)-y(\tau)|^p d\tau \bigg]\\
& \le 2^{p-1}\mathcal N^p \mathcal L^p \left[ \left(\frac{2(p-1)}{\nu p}\right)^{p-1} + c_p  \left(\frac{p-2}{\nu p} \right)^{p/2-1} \right]\\
& \qquad \times \int_{-\infty}^{t} e^{-\frac12\nu p (t-\tau)} \mathbb E |x(\tau)-y(\tau)|^p d\tau\\
&\le 2^{p-1}\mathcal N^p \mathcal L^p \left[ \left(\frac{2(p-1)}{\nu p}\right)^{p-1} + c_p  \left(\frac{p-2}{\nu p} \right)^{p/2-1} \right]
\cdot\frac{2}{\nu p}\cdot \sup_{\tau\in\mathbb R} \mathbb E |x(\tau)-y(\tau)|^p.
\end{align*}
Therefore,
\[
\sup_{t\in\mathbb R}\mathbb E|(\mathcal S x)(t) - (\mathcal S y)(t)|^p \le \theta_p \sup_{t \in\mathbb R} \mathbb E |x(t)-y(t)|^p.
\]
That is, the operator $\mathcal S$ is a contraction mapping on $C_b(\mathbb R,L^p(\mathbb P,H))$. Thus there is a unique
$\xi \in C_b(\mathbb R,L^p(\mathbb P,H))$ satisfying $\mathcal S \xi =\xi$, which is the unique $L^p$ bounded solution of \eqref{eq3.4.4}.
\end{proof}

\begin{remark}\label{RLp}\rm
Note that the contraction constant $\theta_p$ is continuous in $p$ when $p>2$. Furthermore, $c_p=1$ when $p=2$ in Proposition \ref{Lp},  so we have
\[
\lim_{p\to 2^+} \theta_p= \frac{2\mathcal N^2\mathcal L^2}{\nu^2} + \frac{2\mathcal N^2\mathcal L^2}{\nu}.
\]
\end{remark}

\begin{theorem}\label{t3.4.4}
Consider the equation \eqref{eq3.4.4}. Suppose that the following conditions hold:
\begin{enumerate}
\item the semigroup $\{U(t)\}_{t\ge 0}$ acting on the space $H$
is exponentially stable;
\item $F,G\in C(\mathbb R\times  H,  H)$;
\item the functions $F$ and $G$ satisfy the conditions (C1) and
(C2).
\end{enumerate}

Then the following statements hold:
\begin{enumerate}
\item If $\mathcal L< \frac{\nu}{\mathcal N \sqrt{2+\nu}}$, then equation $(\ref{eq3.4.4})$ has a unique solution $\xi \in
C(\mathbb{R},B[0,r]),$ where
\begin{equation}\label{r}
r=\frac{\mathcal N
A_0\sqrt{2+\nu}}{\nu-\mathcal N \mathcal L\sqrt{2+\nu}}
\end{equation}
and
$$B[0,r]:=\{x\in L^{2}(\mathbb P,H):\ ||x||_2\le r\}.$$

\item If $\mathcal L<\frac{\nu}{2\mathcal N\sqrt{1+\nu}}$ and additionally $F,G$ satisfy (C3), then
\begin{enumerate}

\item  $\mathfrak M^{u}_{(F,G)} \subseteq \tilde{\mathfrak M}^{u}_{\xi}$, recalling that
$\tilde{\mathfrak M}^{u}_{\xi}$ means the set of all sequences $\{t_n\}$ such that
$\xi(t+t_n)$ converges in distribution uniformly in $t\in\mathbb R$;

\item  the solution $\xi$ is uniformly comparable in distribution.
\end{enumerate}
\end{enumerate}
\end{theorem}

\begin{proof}
(i). 
Note that $C(\mathbb{R},B[0,r])$ is a complete metric space. Define
an operator
$$
\Phi: C(\mathbb{R},B[0,r]) \to C(\mathbb{R},B[0,r])
$$
as follows. If $\phi\in C(\mathbb{R},B[0,r])$, then we put
$h_1(t):=F(t,\phi(t))$ and $h_2(t):=G(t,\phi(t))$ for any $t\in
\mathbb R$. Since the function $F$ satisfies conditions (C1) and
(C2), we have
\begin{equation}\label{eqH1}
||h_1(t)||_2=||F(t,\phi(t))||_2 \le ||F(t,0)||_2 + \mathcal L ||\phi(t)||_2 \le
A_0+\mathcal L r
\end{equation}
for any $t\in \mathbb R$. Analogically we have
\begin{equation}\label{eqH2}
||h_2(t)||_2 \le A_0+ \mathcal L ||\phi(t)||_2 \le A_0+\mathcal L r
\end{equation}
for $t\in \mathbb R$. According to Theorem \ref{thLS11_2}, the
equation
$$
dz(t)=(Az(t)+ h_1(t))dt + h_2(t)dW(t)
$$
has a unique solution $\psi\in C_b(\mathbb{R},L^{2}(\mathbb
P,H))$. Besides, it obeys the estimate
\begin{equation}\label{eqH3}
||\psi||_{\infty}\leq \frac{\mathcal N}{\nu}(2||h_1||_{\infty}^{2}+\nu
||h_2||_{\infty}^{2})^{1/2}.
\end{equation}
From (\ref{eqH1})-(\ref{eqH3}) and \eqref{r} we have
\begin{equation*}\label{eqH4}
||\psi||_{\infty}\leq \frac{\mathcal N}{\nu}(2(A_0+\mathcal Lr)^2+\nu(A_0+\mathcal Lr)^2)^{1/2}=\frac{\mathcal
N\sqrt{2+\nu}}{\nu}(A_0+\mathcal L r) = r.
\end{equation*}
So $\psi\in C(\mathbb{R},B[0,r])$. Let $\Phi(\phi):=\psi$. It
follows from the above argument that $\Phi$ is well defined.

Let us show that the operator $\Phi$ is a contraction. In
fact, it is easy to note that the function
$\psi_1-\psi_2=\Phi(\phi_1)-\Phi(\phi_2)$ is the unique solution
from  $C_{b}(\mathbb R,L^{2}(\mathbb P,H))$ of the equation
\begin{align*}
du(t)=(Au(t)+F(t,\phi_1(t))-F(t,\phi_2(t)))dt +
(G(t,\phi_1(t))- G(t,\phi_2(t)))dW(t).
\end{align*}
By Theorem \ref{thLS11_2}, we have the following estimate
\begin{align}\label{L2}
||\Phi(\phi_1)-\Phi(\phi_2)||_{\infty}^2 & \leq  \frac{
\mathcal N^2}{\nu^2}\Big{(} 2 \sup\limits_{t\in \mathbb R} \mathbb
E|F(t,\phi_1(t))-F(t,\phi_2(t))|^{2} \nonumber\\
& \qquad  +\nu \sup\limits_{t\in
\mathbb R} \mathbb
 E|G(t,\phi_1(t))-G(t,\phi_2(t))|^{2}\Big{)} \nonumber \\
& \leq \frac{\mathcal N^2
\mathcal L^2 (2+\nu)}{\nu^2}||\phi_1-\phi_2||_{\infty}^2 =: \theta_2
||\phi_1-\phi_2||_{\infty}^2.
\end{align}
By the assumption on $\mathcal L$ we have
$$
\theta_2 = \frac{\mathcal N^2
\mathcal L^2 (2+\nu)}{\nu^2} < \frac{\mathcal N^2
(2+\nu)}{\nu^2} \cdot \frac{\nu^2}{\mathcal N^2 (2+\nu)}=1,
$$
so $\Phi$ is a contraction.  Consequently, there exists a unique
function $\xi\in C (\mathbb R, B[0,r])$ such that $\Phi(\xi)=\xi$.

(ii)-(a).
Let $\{t_n\}\in \mathfrak
M_{(F,G)}^{u}$. Then there exists
$(\tilde{F},\tilde{G})\in H(F,G)$ such that for any $r>0$
\begin{equation}\label{eqUFG1}
\sup\limits_{t\in \mathbb R,\ |x|\le r} |F(t+t_n,x)-\tilde{F}(t,x)| \to 0
\end{equation}
and
\begin{equation}\label{eqUFG2}
\sup\limits_{t\in\mathbb R,\ |x|\le r} |G(t+t_n,x)-\tilde{G}(t,x)| \to 0
\end{equation}
as $n\to \infty$. Consider equations
\begin{equation}\label{eqUFGn}
dx(t)=(Ax(t)+F^{t_n}(t,x(t)))dt + G^{t_n}(t,x(t))dW(t) \ (n\in
\mathbb N)
\end{equation}
and
\begin{equation}\label{eqUFG0}
dx(t)=(Ax(t)+\tilde{F}(t,x(t)))dt + \tilde{G}(t,x(t))dW(t).
\end{equation}
Since the functions $(F^{t_n},G^{t_n})$ ($n\in\mathbb N$) and
$(\tilde{F},\tilde{G})$ satisfy conditions (C1) and (C2) (see Remark
\ref{remL10}), by the first part of the theorem equation
(\ref{eqUFGn}) (respectively, equation (\ref{eqUFG0})) has a unique
solution $\xi_{n}\in C(\mathbb R, B[0,r])$ (respectively,
$\tilde{\xi}\in C (\mathbb R, B[0,r])$). We will show that
$\{\xi_{n}(t)\}$ converges, in $L^2$ norm, to $\tilde{\xi}(t)$
uniformly in $t\in\mathbb R$. To this end we note that $\xi_{n}$
($n\in\mathbb N$) is the unique solution from $C(\mathbb R, B[0,r])$
of equation
\begin{equation*}\label{eqULn}
dx(t)=(Ax(t)+h_{n}(t))dt + g_{n}(t)dW(t) \ (n\in \mathbb N),
\end{equation*}
where
$(h_n(t),g_n(t)):=(F^{t_n}(t,\xi_{n}(t)),G^{t_n}(t,\xi_{n}(t))$ for
$t\in \mathbb R$ and $n\in\mathbb N$ and, respectively,
$\tilde{\xi}$ is the unique solution from $C (\mathbb R, B[0,r])$ of equation
\begin{equation*}\label{eqULn1}
dx(t)=(Ax(t)+\tilde{h}(t))dt + \tilde{g}(t)dW(t) \ (n\in \mathbb N),
\end{equation*}
where
$(\tilde{h}(t),\tilde{g}(t)):=(\tilde{F}(t,\tilde{\xi}(t)),\tilde{G}(t,\tilde{\xi}(t))$
for $t\in \mathbb R$. It is easy to check that
$\phi_{n}:=\xi_{n}-\tilde{\xi}$ is the unique solution from $C
(\mathbb R, B[0,2r])$ of the equation
\begin{equation}\label{eqUL1n}
dx(t)=(Ax(t)+(h_{n}(t)-\tilde{h}(t)))dt + (g_{n}(t)-\tilde{g}(t)
)dW(t) \ (n\in \mathbb N),
\end{equation}
where $h_{n}-\tilde{h}, g_{n}-\tilde{g} \in C_{b}(\mathbb
R,L^{2}(\mathbb R,H))$. In virtu of Theorem \ref{thLS11_2} (item
(ii)) we have
\begin{equation}\label{eqUBL51}
||\phi_{n}||^2_{\infty}\le \frac{\mathcal
N^2}{\nu^2}\big{(}2||h_{n}-\tilde{h}||^{2}_{\infty}+ \nu
||g_{n}-\tilde{g}||^{2}_{\infty}\big{)}.
\end{equation}

Taking into consideration that the functions $(F^{t_n},G^{t_n})$
($n\in\mathbb N$) and $(\tilde{F},\tilde{G})$ satisfy conditions
(C1) and (C2), and $\xi_{n},\tilde{\xi}\in C(\mathbb
R, B[0,r])$ ($n\in\mathbb N$) we have
\begin{align}\label{eqUBL52}
\mathbb E|h_{n}(\tau)-\tilde{h}(\tau))|^2 & =\mathbb
E|F^{t_n}(\tau,\xi_n(\tau))-F^{t_n}(\tau,\tilde{\xi}(\tau))+F^{t_n}(\tau,\tilde{\xi}(\tau))
-\tilde{F}(\tau,\tilde{\xi}(\tau))|^2 \nonumber \\
&\le
 2(\mathbb
E|F^{t_n}(\tau,\xi_n(\tau))-F^{t_n}(\tau,\tilde{\xi}(\tau))|^2
+\mathbb E|F^{t_n}(\tau,\tilde{\xi}(\tau))
-\tilde{F}(\tau,\tilde{\xi}(\tau))|^2) \nonumber\\
&\le  2(\mathcal L^2\mathbb E|\xi_n(\tau)-\tilde{\xi}(\tau)|^2 +
\sup\limits_{\tau\in\mathbb R}\mathbb E|F^{t_n}(\tau,\tilde{\xi}(\tau))
-\tilde{F}(\tau,\tilde{\xi}(\tau))|^2) \nonumber \\
& \le
2(\mathcal L^2||\phi_{n}||^2_{\infty} + \sup\limits_{\tau\in\mathbb R}\mathbb Ea_{n,\tau}^2),
\end{align}
where
\[
a_{n,\tau}:= |F^{t_n}(\tau,\tilde{\xi}(\tau))
-\tilde{F}(\tau,\tilde{\xi}(\tau))|.
\]
Using the same arguments we have
\begin{eqnarray}\label{eqUBL54}
\mathbb E|g_{n}(\tau)-\tilde{g}(\tau)|^2\le 2(\mathcal L^2||\phi_{n}||^2_{\infty} + \sup\limits_{\tau\in\mathbb R}\mathbb E b_{n,\tau}^2)
\end{eqnarray}
with
\[
b_{n,\tau}:= |G^{t_n}(\tau,\tilde{\xi}(\tau))-\tilde{G}(\tau,\tilde{\xi}(\tau))|.
\]
From (\ref{eqUBL51})-(\ref{eqUBL54}) we obtain
\begin{equation*}\label{eqUBL55}
||\phi_{n}||^2_{\infty}  \le \frac{\mathcal N^2}{\nu^2}\Big{[}
4(\mathcal L^2||\phi_{n}||^2_{\infty} + \sup\limits_{\tau\in\mathbb R}\mathbb Ea_{n,\tau}^2) +
2\nu(\mathcal L^2||\phi_{n}||^2_{\infty} + \sup\limits_{\tau\in\mathbb R}\mathbb Eb_{n,\tau}^2) \Big{]}.
\end{equation*}
Consequently,
\begin{eqnarray}\label{eqUBL56}
\Big{(}1-\frac{2\mathcal N^2\mathcal
L^2}{\nu^2}(2+\nu)\Big{)}||\phi_{n}||^2_{\infty}\le \frac{4\mathcal
N^2}{\nu^2} \sup\limits_{\tau\in\mathbb R}\mathbb Ea_{n,\tau}^2 +
\frac{2\mathcal N^2}{\nu} \sup\limits_{\tau\in\mathbb R}\mathbb E
b_{n,\tau}^2.
\end{eqnarray}
By our assumption on $\mathcal L$, the coefficients of $||\phi_{n}||^2_{\infty}$ is positive.

We note by \eqref{L2} and Remark \ref{remL10} that, for $p=2$, the
contraction constant $\theta_2$ for the equation \eqref{eqUFG0} is
\[
\theta_2= \frac{2\mathcal N^2\mathcal L^2}{\nu^2}  +  \frac{\mathcal N^2\mathcal L^2}{\nu}.
\]
Comparing to Remark \ref{RLp}, we have
\[
\lim_{p\to 2^+} \theta_p = \theta_2 + \frac{\mathcal N^2\mathcal L^2}{\nu}.
\]
We also note that $\lim_{p\to 2^+} \theta_p <1$ if and only if
\begin{equation}\label{bs}
\mathcal L<\frac{\nu}{\mathcal{N}\sqrt{2(1+\nu)}},
\end{equation}
which is satisfied by our assumption on $\mathcal L$.
So it follows from Proposition \ref{Lp} that
\eqref{eqUFG0} admits a unique $L^p$-bounded solution for some $p>2$.
This $L^p$-bounded solution is exactly the unique $L^2$-bounded solution $\tilde\xi$ of \eqref{eqUFG0}. So the family
\[
\{ |\tilde\xi(\tau)|^2: \tau\in\mathbb R\}
\]
is uniformly integrable, and hence by conditions (C1) and (C2) the families
\[
\{a_{n,\tau}^2: n\in\mathbb N, \tau\in \mathbb R\} \quad \hbox{and}\quad
\{b_{n,\tau}^2: n\in\mathbb N, \tau\in \mathbb R\}
\]
are uniformly integrable. This together with \eqref{eqUFG1} and
\eqref{eqUFG2} implies: taking limit in (\ref{eqUBL56}), we obtain
the required result, i.e. $\xi_{n}(t)\to \tilde{\xi}(t)$  uniformly
in $t\in\mathbb R$ in $L^2$-norm.


Since $L^2$ convergence implies convergence in distribution, we have
$\xi_n(t)\to \tilde{\xi}(t)$ in distribution uniformly on $\mathbb R$.
On the other hand, $\xi(t+t_n)$ satisfies the equation
$$
\xi(t+t_n)= \int _{-\infty}^{t}U(t-\tau)F(\tau +t_n,\xi(\tau
+t_n))d\tau +\int _{-\infty}^{t}U(t-\tau)G(\tau +t_n,\xi(\tau
+t_n))d\tilde{W}_n(\tau),
$$
with $\tilde W_n(t)=W(t+t_n)-W(t_n)$. Note that $\tilde W_n(\cdot)$ is also a standard
Brownian motion with the same distribution as $W(\cdot)$, so $\xi_{n}(t)$ and $\xi(t+t_n)$ share the same distribution on $H$.
This implies $\xi(t+t_n)\to \tilde{\xi}(t)$ in distribution uniformly in $t\in
\mathbb R$. Thus we have $\{t_n\}\in \mathfrak {\tilde{M}}^{u}_{\xi}$.

(ii)-(b). 
Let $\{t_n\}\in \mathfrak M_{(F,G)}$. Then there exists
$(\tilde{F},\tilde{G})\in H(F,G)$ such that for any $r,l>0$
\begin{equation}\label{eqFG1}
\sup\limits_{|t|\le l,\ |x|\le r} |F(t+t_n,x)-\tilde{F}(t,x)| \to 0
\end{equation}
and
\begin{equation}\label{eqFG2}
\sup\limits_{|t|\le l,\ |x|\le r} |G(t+t_n,x)-\tilde{G}(t,x)| \to 0
\end{equation}
as $n\to \infty$. Like what we did in the proof of (ii)-(a): let
$\xi_n$ and $\tilde\xi$ be the unique bounded solutions of the shift
equation and the limit equation respectively, and still denote
$\phi_n=\xi_n-\tilde\xi$. To finish the proof, it suffices to show
$\phi_n\to 0$ in the space $C(\mathbb R,L^{2}(\mathbb R,H))$, i.e.
$\lim_{n\to\infty}\max_{|t|\le L} \mathbb E|\phi_n(t)|^2=0$ for any
$L>0$.

Since $\phi_n$ is the unique bounded solution of equation \eqref{eqUL1n}, by the Cauchy-Schwarz inequality and
It\^o's isometry property we have
\begin{align}\label{eqBL51}
\mathbb E|\phi_n(t)|^2 & \le 2\mathcal N^2\Big{(}
\frac{1}{\nu}\int_{-\infty}^{t}e^{-\nu (t-\tau)}\mathbb
E|h_n(\tau)-\tilde{h}(\tau)|^2d\tau \\
&\qquad +
 \int_{-\infty}^{t}e^{-2\nu
(t-\tau)}\mathbb E|g_n(\tau)-\tilde{g}(\tau)|^2d\tau \Big{)}.
\nonumber
\end{align}
By \eqref{eqUBL52} we have
\begin{equation}\label{eqBL52}
\mathbb E|h_{n}(\tau)-\tilde{h}(\tau)|^2  \le
 2(\mathcal L^2\mathbb E|\phi_n(\tau)|^2 + \mathbb E a_{n,\tau}^2).
\end{equation}
Similar to \eqref{eqH1}, we have
\begin{align*}\label{eqBL53}
\mathbb E|h_n(\tau)-\tilde{h}(\tau)|^2 &=\mathbb
E|F^{t_n}(\tau,\xi_n(\tau)) -\tilde{F}(\tau,\tilde{\xi}(\tau))|^2
\nonumber
\\
& \le 2(\mathbb E|F^{t_n}(\tau,\xi_n(\tau))|^2+\mathbb
E|\tilde{F}(\tau,\tilde{\xi}(\tau))|^2)  \\
& \le
4(A_0+\mathcal L r)^2\nonumber
\end{align*}
for any $\tau\in\mathbb R$ and, consequently,
\begin{equation}\label{eqBL531}
||h_n-\tilde{h}||_{\infty}^2\le
4(A_0+\mathcal L r)^2
\end{equation}
for any $n\in\mathbb N$.

Using the same arguments as above we have
\begin{equation}\label{eqBL54}
\mathbb E|g_{n}(\tau)-\tilde{g}(\tau)|^2\le 2(\mathcal L^2\mathbb E|\phi_n(\tau)|^2 +  \mathbb
E b_{n,\tau}^2)
\end{equation}
and
\begin{eqnarray}\label{eqBL55}
 ||g_n-\tilde{g}||_{\infty}^2\le 4(A_0+\mathcal L r)^2
\end{eqnarray}
for any $n\in\mathbb N$.

From (\ref{eqBL51}), (\ref{eqBL52}), (\ref{eqBL54}) and taking into account that
$e^{-2\nu (t-\tau)}\le e^{-\nu (t-\tau)}$ ($t\ge \tau$), we obtain
\begin{align*}
 \mathbb E|\phi_n(t)|^2 & \le \left( \frac{4\mathcal N^2\mathcal L^2}{\nu} +
4\mathcal
N^2\mathcal L^2\right)\int_{-\infty}^{t}e^{-\nu (t-\tau)}\mathbb E|\phi_n(\tau)|^{2}d\tau \nonumber \\
&\qquad + 4\mathcal N^2\int_{-\infty}^{t}e^{-\nu
(t-\tau)}\left(\frac{1}{\nu}\mathbb E a_{n,\tau}^2+\mathbb E b_{n,
\tau}^2\right)d\tau.
\end{align*}
By Lemma \ref{lB} we have
\begin{align}\label{eqBL57}
\max\limits_{|t|\le L}\mathbb E|\phi_n(t)|^2 & \le 4\mathcal
N^2\frac{e^{kL}e^{-kl}}{k}\sup\limits_{t\in\mathbb
R}\left(\frac{1}{\nu}\mathbb E a_{n,t}^2+\mathbb E b_{n,t}^2\right)\nonumber\\
&\qquad +4\mathcal N^2\frac{1-e^{-kL}e^{-kl}}{k}\max\limits_{|t|\le
l}\left(\frac{1}{\nu}\mathbb E a_{n,t}^2+ \mathbb E
b_{n,t}^2\right),
\end{align}
where
$$
k:=\nu -( \frac{4\mathcal N^2}{\nu}\mathcal L^2 + 4\mathcal
N^2\mathcal L^2)>0
$$
by the assumption $\mathcal L < \frac{\nu}{2\mathcal
N\sqrt{1+\nu}}$.

Let now $\{l_n\}$ be a sequence of positive numbers such that
$l_n\to +\infty$ as $n\to \infty$. According to inequality
(\ref{eqBL531}), (\ref{eqBL55}) and (\ref{eqBL57}) we obtain
\begin{align}\label{eqBL58}
\max\limits_{|t|\le L}\mathbb E|\phi_n(t)|^2 & \le \frac{16\mathcal
N^2e^{kL}e^{-kl_n}}{k}(\frac{1}{\nu} +1)(A_0+\mathcal L r)^2 \nonumber \\
& \qquad + \frac{4\mathcal
N^2(1-e^{-kL}e^{-kl_n})}{k}\max\limits_{|t|\le
l_n}\left(\frac{1}{\nu}\mathbb E a_{n,t}^2 + \mathbb E b_{n,t}^2
\right).
\end{align}
By Remark \ref{remD1}-(iii), passing to limit in (\ref{eqBL58}) as
$n\to \infty$ we obtain for any $L>0$
\begin{equation*}\label{eqBL59}
\lim\limits_{n\to \infty}\max\limits_{|t|\le L}\mathbb
E|\phi_{n}(t)|^2=0
\end{equation*}
by \eqref{eqFG1}, \eqref{eqFG2} and the uniform integrability of the families
$\{a_{n,\tau}^2: n\in\mathbb N, \tau\in \mathbb R\}$ and $\{b_{n,\tau}^2: n\in\mathbb N, \tau\in \mathbb R\}$.
That is, $\xi_{n}\to \tilde{\xi}$ as $n\to \infty$ in the space
$C(\mathbb R,L^{2}(\mathbb P,H))$.
So we have
$\xi_n(t)\to \tilde{\xi}(t)$ in distribution uniformly in $t\in
[-L,L]$ for any $L>0$. Since $\xi_n (t)$ and $\xi(t+t_n)$ share the same distribution, $\xi(t+t_n)\to
\tilde{\xi}(t)$ in distribution uniformly in $t\in [-L,L]$ for all
$L>0$. Thus we have
$\{t_n\}\in \mathfrak {\tilde{M}}_{\xi}$, and hence $\xi$ is uniformly comparable in distribution.
The theorem is completely proved.
\end{proof}

\begin{coro}\label{corAA1}
Assume that the conditions of Theorem
\ref{t3.4.4} hold.
\begin{enumerate}
\item
If the functions $F$ and $G$ are jointly stationary (respectively,
$\tau$--periodic, quasi-periodic with the spectrum of frequencies
$\nu_1,\nu_2,\dots,\nu_k$, Bohr almost periodic, Bohr almost automorphic,
Birkhoff recurrent, Lagrange stable, Levitan almost periodic, almost
recurrent, Poisson stable) in $t\in\mathbb R$ uniformly with respect
to $x\in  H$ on every bounded subset, then so is the unique bounded
solution $\xi$ of equation (\ref{eq3.4.4}) in distribution.

\item
If $F$ and $G$ are jointly pseudo-periodic (respectively, pseudo-recurrent)
and $F$ and $G$ are jointly Lagrange stable, in $t\in\mathbb R$ uniformly with respect
to $x\in  H$ on every bounded subset, then the
unique bounded solution $\xi$ of (\ref{eq3.4.4}) is pseudo-periodic
(respectively, pseudo-recurrent) in distribution.
\end{enumerate}
\end{coro}

\begin{proof} This statement follows from Theorems \ref{th1},
\ref{t3.4.4} and Remark \ref{remBUC}.
\end{proof}



\section{Convergence in semi-linear SDEs}

In this section we consider the stochastic differential equation
\begin{equation}\label{eq3.4.40}
dx(t)=(Ax(t)+F(t,x(t)))dt + G(t,x(t))dW(t),
\end{equation}
where $F,G\in C(\mathbb{R}\times  H,  H)$ and the linear operator
$A$ is an infinitesimal generator which generates a
$C_0$-semigroup $\{U(t)\}_{t\ge 0}$, which is exponentially stable.

\begin{definition} \rm
An $\mathcal F_{t}$-adapted
processes $\{x(t)\}_{t\ge t_0}$ is said to be a mild solution
of equation (\ref{eq3.4.40}) with initial value $x(t_0)=x_0$ ($t_0\in
\mathbb R$) if it satisfies the
stochastic integral equation
\begin{align*}
x(t) &= U(t-t_0)x_0+
\int_{t_0}^{t}U(t-s) F(s,x(s)) ds \\
&\quad +
\int_{t_0}^{t}U(t-s) G(s,x(s)) dW(s)
\end{align*}
for $t\ge t_0$.
\end{definition}

\begin{theorem}\label{thC1}
Consider the equation \eqref{eq3.4.40}.
Suppose that the following conditions hold:
\begin{enumerate}
\item the semigroup $\{U(t)\}_{t\ge 0}$ acting on the space $H$ is exponentially stable;
\item $F,G\in C(\mathbb R\times  H,  H)$ are locally Lipschitz in $x\in H$;
\item there exist two positive constants $A_0, M$ such that $|F(t,x)|,|G(t,x)|\leq A_0 +M|x| $ for all $x\in  H$ and
$t\in\mathbb{R}$;
\item  $M<\frac{\nu}{\mathcal N \sqrt{6(\nu+1)}}$.
\end{enumerate}

Then for any initial value $x_0$ with $\mathbb E |x_0|^2<\infty$ we have
\begin{align}\label{eqC1_00}
&\mathbb E|x(t;t_0,x_0)|^{2}\nonumber \\
&\le 3\mathcal N^2\left[\mathbb E |x_0|^{2}-\frac{2A_0^2(\nu+1)}{\nu^2 -6\mathcal
N^{2} M^{2}(\nu+1)}\right] \exp\{-[\nu - 6\mathcal N^{2} M^2(1+1/\nu)](t-t_0)\} \nonumber\\
&\qquad + \frac{6\mathcal N^{2}A_0^2 (\nu+1)}{\nu^2 - 6\mathcal N^{2}M^{2} (\nu+1)}
\end{align}
for any $t\ge t_0$, where $x(t;t_0,x_0)$ denotes the solution of the equation (\ref{eq3.4.40})
passing through $x_0$ at the initial moment $t_0$.
\end{theorem}

\begin{proof}
Since
\begin{align*}
 x(t;t_0,x_0) &=U(t-t_0)x_0+\int_{t_0}^{t}U(t-s)F(s,x(s;t_0,x_0))ds \\
& \quad +\int_{t_0}^{t}U(t-s)G(s,x(s;t_0,x_0))dW(s)
\end{align*}
for any $t\ge t_0$, by the Cauchy-Schwarz inequality and It\^o's isometry property we have
\begin{align}\label{eqC1_3}
 \mathbb E|x(t;t_0,x_0)|^2  = & ~\mathbb E\bigg|U(t-t_0)x_0+\int_{t_0}^{t}U(t-s)F(s,x(s;t_0,x_0))ds\nonumber \\
& \quad +\int_{t_0}^{t}U(t-s)G(s,x(s;t_0,x_0))dW(s)\bigg|^2 \nonumber\\
 \le & ~
3\bigg(\mathbb E|U(t-t_0)x_0|^2+ \mathbb E\left|\int_{t_0}^{t}U(t-s)F(s,x(s;t_0,x_0))ds\right|^2  \nonumber \\
&
\quad + \mathbb E\left|\int_{t_0}^{t}U(t-s)G(s,x(s;t_0,x_0))dW(s)\right|^2\bigg) \nonumber \\
\le & ~
3\bigg[\mathcal N^2e^{-2\nu (t-t_0)}\mathbb E|x_0|^2+ \mathbb E\left| \int_{t_0}^{t} U(t-s)F(s,x(s;t_0,x_0))ds \right|^2 \nonumber \\
& \quad  + \int_{t_0}^{t}\mathbb E|U(t-s)G(s,x(s;t_0,x_0))|^2ds\bigg] \nonumber \\
\le & ~ 3\bigg[\mathcal N^2e^{-2\nu (t-t_0)}\mathbb E|x_0|^2+
\frac1\nu\cdot\int_{t_0}^{t}\mathcal N^2e^{-\nu (t-s)}\mathbb E|F(s,x(s;t_0,x_0))|^2ds  \nonumber \\
&
\quad + \int_{t_0}^{t}\mathcal N^2e^{-2\nu (t-s)}\mathbb E|G(s,x(s;t_0,x_0))|^2ds\bigg] \nonumber \\
\le &~  3\mathcal N^2e^{-\nu t} \left[e^{\nu t_0}\mathbb E |x_0|^2+
2(1+\frac1\nu)\int_{t_0}^{t}e^{\nu s}(A_0^2 + M^2\mathbb
E|x(s;t_0,x_0)|^2)ds \right].
\end{align}
Denote
\begin{equation*}\label{eqC1_40}
u(t):=e^{\nu t} \mathbb E|x(t;t_0,x_0)|^2,\quad \hbox{for } t\ge t_0.
\end{equation*}
Then it follows from (\ref{eqC1_3}) that
\begin{align}\label{eqC1_41}
u(t) & \le 3\mathcal N^2 \left[e^{\nu t_0}\mathbb E|x_0|^{2}+ 2(1+\frac1\nu) \int_{t_0}^{t} (A_0^2 e^{\nu s}+M^{2}u(s))ds\right] \nonumber \\
&= 3\mathcal N^2 e^{\nu t_0}\mathbb E|x_0|^{2}+\frac{6\mathcal
N^2A_0^2}{\nu^2}(\nu+1)(e^{\nu t}-e^{\nu t_0})+  6 \mathcal
N^2M^{2}(1+\frac1\nu)\int_{t_0}^{t}u(s)ds.
\end{align}
Along with inequality (\ref{eqC1_41}) we consider the equation
\begin{equation*}
 v(t)=
3\mathcal N^2 e^{\nu t_0}\mathbb E|x_0|^{2}+\frac{6\mathcal
N^2A_0^2}{\nu^2}(\nu+1)(e^{\nu t}-e^{\nu t_0})+  6 \mathcal
N^2M^{2}(1+\frac1\nu)\int_{t_0}^{t}v(s)ds;
\end{equation*}
that is, $v(t)$ satisfies the equation
\begin{equation*}\label{eqC1_42}
v'(t) = 6 \mathcal N^2M^{2}(1+\frac1\nu) v(t) + {6\mathcal
N^2A_0^2}(1+\frac1\nu)e^{\nu t}
\end{equation*}
with initial condition $v(t_0)= 3\mathcal N^2 e^{\nu t_0}\mathbb E|x_0|^{2}$.
Solving this equation for $v(t)$ we get
\begin{align*}\label{eqC1_5}
v(t)& = 3\mathcal N^2 e^{\alpha(t-t_0) +\nu t_0}\mathbb E|x_0|^{2} + \frac{\beta}{\nu-\alpha}[e^{\nu t} - e^{\alpha(t-t_0)+\nu t_0}]\nonumber\\
&= \left(3\mathcal N^2 \mathbb E|x_0|^{2} - \frac{\beta}{\nu-\alpha}\right) e^{\alpha(t-t_0) +\nu t_0} + \frac{\beta}{\nu-\alpha}e^{\nu t},
\end{align*}
where
\[
\alpha:= 6\mathcal N^2 M^2 (1+\frac1\nu) \quad \hbox{and} \quad \beta:= 6\mathcal N^2 A_0^2 (1+\frac1\nu).
\]
The comparison principle then implies that $u(t)\le v(t)$ for all $t\ge t_0$, so it follows from the
definition of $u(t)$ that for $t\ge t_0$ we have
\begin{align*}
\mathbb E|x(t;t_0,x_0)|^{2}\le \left(3\mathcal N^2 \mathbb E|x_0|^{2}
- \frac{\beta}{\nu-\alpha}\right) e^{-(\nu-\alpha)(t-t_0)} + \frac{\beta}{\nu-\alpha},
\end{align*}
which is just \eqref{eqC1_00}. The proof is complete.
\end{proof}

Note that condition (iv) in Theorem \ref{thC1} implies $\nu>\alpha=6\mathcal N^2 M^2 (1+1/\nu)$. So we have the following

\begin{coro}\label{UB}
Under the conditions of Theorem \ref{thC1}, for arbitrary $\varepsilon >0$ and $r>0$ there exists a positive
number $T(\varepsilon,r)$ such that
$$
\mathbb E|x(t;t_0,x_0)|^2  <   \frac{{6}\mathcal N^2
A_0^2(\nu+1)}{\nu^2 - 6\mathcal N^{2} M^{2}(\nu+1)} +\varepsilon
$$
for all $||x_0||_2 \le r$ and $t\ge t_0+T(\varepsilon,r)$. In other words, we have
\begin{equation*}
\limsup\limits_{t\to \infty} \mathbb E|x(t;t_0,x_0)|^2 \le
 \frac{{6}\mathcal N^2 A_0^2(\nu+1)}{\nu^2 - 6\mathcal N^{2}
M^{2}(\nu+1)}
\end{equation*}
uniformly with respect to $x_0$ on every bounded subset of
$L^2(\mathbb P,H)$.
\end{coro}

\begin{theorem}\label{thC2}
Consider the equation \eqref{eq3.4.40}.
Suppose that the following conditions hold:
\begin{enumerate}
\item the semigroup $\{U(t)\}_{t\ge 0}$ acting on the space $H$
is exponentially stable;
\item $F,G\in C(\mathbb R\times  H,  H)$ are globally Lipschitz in $x\in H$ and $Lip(F),Lip(G)\le\mathcal L$;
\item there exists a positive constant $A_0$ such that $|F(t,0)|,|G(t,0)|\leq A_0$ for all $t\in\mathbb{R}$;
\item $\mathcal L< \frac{\nu}{\mathcal N\sqrt{3(\nu+1)}}$.
\end{enumerate}

Then the following statements hold:
\begin{enumerate}
\item for any $t\ge t_0$ and $x_1,x_2\in L^{2}(\mathbb P,H)$,
\begin{align}\label{eqC1_001}
& \mathbb E|x(t;t_0,x_1)-x(t;t_0,x_2)|^2 \nonumber\\
\le &~ 3\mathcal N^2 \exp\left\{-[\nu
-3(1+\frac1\nu) \mathcal N^2 \mathcal L^2](t- t_0)\right\} \mathbb
 E|x_1-x_2|^2;
\end{align}

\item
equation \eqref{eq3.4.40} has a unique solution $\varphi \in
C_{b}(\mathbb R,L^2(\mathbb P,H))$ which is globally asymptotically
stable and
\begin{align}\label{eqC1_01}
 &\mathbb E|x(t;t_0,x_0)-\varphi(t)|^2 \nonumber\\
 &\le 3\mathcal N^2 \exp\left\{-\Big[\nu
-3(1+\frac1\nu) \mathcal N^2 \mathcal L^2\Big](t- t_0)\right\} \mathbb
 E|x_0-\varphi(t_0)|^2
\end{align}
for any $t\ge t_0$ and $x_0\in L^{2}(\mathbb P,H)$.



\end{enumerate}
\end{theorem}

\begin{proof}
(i). Denote by $\omega(t):=x(t;t_0,x_1)-x(t;t_0,x_2)$ for any $t\ge t_0$. Since
\begin{align*}
x(t;t_0,x_{i}) &= U(t-t_0)x_{i}+\int_{t_0}^{t}U(t-s)F(s,x(s;t_0,x_{i}))ds \\
&\quad +\int_{t_0}^{t}U(t-s)G(s,x(s;t_0,x_{i}))dW(s)
\end{align*}
for $i=1,2$, we have
\begin{align*}
 \omega(t) & =U(t-t_0)(x_1-x_2) + \int_{t_0}^{t} U(t-s) [F(s,x(s;t_0,x_{1})) - F(s,x(s;t_0,x_{2}))]ds \\
 &\quad + \int_{t_0}^{t} U(t-s) [G(s,x(s;t_0,x_{1})) - G(s,x(s;t_0,x_{2}))]dW(s).
\end{align*}
Consequently, 
\begin{align}\label{eqC2.2}
\mathbb E|\omega(t)|^2
& \le  3\bigg(\mathbb E|U(t-t_0)(x_1-x_2)|^2 \nonumber\\
&\quad +
\mathbb E\bigg| \int_{t_0}^{t} U(t-s) [F(s,x(s;t_0,x_{1})) - F(s,x(s;t_0,x_{2}))]ds\bigg|^2 \nonumber \\
& \quad +\mathbb E\left|\int_{t_0}^{t} U(t-s) [G(s,x(s;t_0,x_{1})) - G(s,x(s;t_0,x_{2}))]dW(s)\right|^2\bigg) \nonumber \\
& \le 3\bigg(\mathcal N^2 e^{-2\nu (t-t_0)}\mathbb E|x_1-x_2|^2 \nonumber \\
&\quad   + \mathcal N^2\int_{t_0}^{t} e^{-\nu (t-s)} ds
\cdot \int_{t_0}^{t} e^{-\nu (t-s)}\mathbb E|F(s,x(s;t_0,x_{1})) - F(s,x(s;t_0,x_{2}))|^2ds  \nonumber \\
&
\quad +\int_{t_0}^{t}\mathcal N^2 e^{-2\nu (t-s)}\mathbb E|G(s,x(s;t_0,x_{1})) - G(s,x(s;t_0,x_{2}))|^2ds\bigg) \nonumber \\
& \le
3\bigg(\mathcal N^2 e^{-2\nu (t-t_0)}\mathbb E|x_1-x_2|^2 \nonumber\\
& \quad +  (1+\frac1\nu)\mathcal N^2 \mathcal L^2 \int_{t_0}^{t} e^{-\nu (t-s)}
        \mathbb E|x(s;t_0,x_{1}) - x(s;t_0,x_{2})|^2ds\bigg)    \nonumber \\
&\le 3\mathcal N^2 e^{-\nu (t-t_0)}\mathbb E|x_1-x_2|^2 + 3
(1+\frac1\nu) \mathcal N^2 \mathcal L^2 \int_{t_0}^{t}e^{-\nu
(t-s)}\mathbb E|\omega(s)|^2ds.
\end{align}

Set $u(t):=e^{\nu t}\mathbb E|\omega(t)|^2$ for $t\ge t_0$, then from (\ref{eqC2.2}) we get
\begin{equation}\label{eqC2.3}
u(t)\le 3\mathcal N^2 e^{\nu t_0}\mathbb E|x_1-x_2|^2 +
3 (1+\frac1\nu) \mathcal N^2 \mathcal L^2 \int_{t_0}^{t}u(s)ds.
\end{equation}
Along with inequality (\ref{eqC2.3}) we consider the equation
\begin{equation*}\label{eqC2.4}
v(t)= 3\mathcal N^2 e^{\nu t_0} \mathbb E |x_1-x_2|^2 +  3 (1+\frac1\nu) \mathcal N^2 \mathcal L^2 \int_{t_0}^{t}v(s)ds.
\end{equation*}
Solving this equation for $v(t)$ we obtain
$$
v(t)=3\mathcal N^2 e^{\nu t_0}\mathbb E|x_1-x_2|^2 \exp\left\{3 (1+\frac1\nu) \mathcal N^2 \mathcal L^2(t-t_0)\right\}.
$$
The comparison principle then implies that  $u(t)\le v(t)$,
i.e.
\begin{equation*}\label{C2.5}
u(t)\le 3\mathcal N^2 e^{\nu t_0}\mathbb E|x_1-x_2|^2 \exp\left\{3 (1+\frac1\nu) \mathcal N^2 \mathcal L^2(t-t_0)\right\},
\quad\hbox{for } t\ge t_0
\end{equation*}
and consequently by the definition of $u(t)$ we get
\begin{equation*}\label{C2.6}
\mathbb E|x(t;t_0,x_1)-x(t;t_0,x_2)|^2\le 3\mathcal N^2 \mathbb E|x_1-x_2|^2 \exp\left\{-\Big[\nu
-3(1+\frac1\nu) \mathcal N^2 \mathcal L^2\Big](t- t_0)\right\}
\end{equation*}
for any $t\ge t_0$.

(ii). By the proof of Theorem \ref{t3.4.4} (see \eqref{bs} and the
paragraph following it), equation \eqref{eq3.4.40} admits a unique
bounded solution $\varphi\in C_b(\mathbb R, L^2(\mathbb P, H))$
under the condition \eqref{bs}, which is met under the current
condition (iv).

To establish inequality (\ref{eqC1_01}) we note that $\varphi(t)=x(t;t_0,\varphi(t_0))$ for any $t\ge t_0$.
Applying \eqref{eqC1_001} we obtain inequality (\ref{eqC1_01}).
\end{proof}

\section{Applications}

In this section, we illustrate our theoretical results by two examples.

\begin{example}\rm
Consider an ordinary differential equation perturbed
by white noise:
\begin{align}\label{ex1}
d y = & \left( -5y + \frac{\cos t+ \sin\sqrt{2}t}{4+\cos\sqrt{3}t}\cdot\frac{y}{y^2+1}\right) d t
+ \frac12 y\sin \Big{(}\frac{1}{2+\cos t+\cos \sqrt2\, t}\Big{)} d W\\
=&:  (Ay + f(t,y)) d t + g(t,y)d W,\nonumber
\end{align}
where $W$ is a one-dimensional two-sided Brownian motion.
It is clear that $A$ generates an exponentially stable semigroup on $\R$ with $\mathcal N=1$ and $\nu=5$. Note that $f$ is quasi-periodic
in $t$ and $g$ is Levitan almost periodic in $t$, uniformly w.r.t $y$ on any bounded subset of $\mathbb R$, so
$f,g$ are jointly Levitan almost periodic. The Lipschitz constants of $f,g$
satisfy $\max\{Lip(f),Lip(g)\}\le 2/3$, so the conditions of Theorems \ref{t3.4.4}, \ref{thC1} and \ref{thC2} are met.

Since the coefficients satisfy both Lipschitz and global linear growth conditions, it follows
that the equation \eqref{ex1} admits global in time solutions.
By Theorem \ref{t3.4.4}, \eqref{ex1} admits a unique $L^2$-bounded mild solution;
furthermore, this unique  $L^2$-bounded solution is Levitan almost periodic in distribution by Corollary \ref{corAA1}. By Theorem \ref{thC2},
this Levitan almost periodic in distribution solution is globally asymptotically stable in square-mean sense.
By Corollary \ref{UB}, all the solutions of \eqref{ex1} with $L^2$-initial value are bounded by a constant after
sufficiently long time.

If $f$ remains unchanged but $g(t,y) = y(\sin t +\cos \sqrt2 t)/4$, then $g$ is quasi-periodic in $t$, uniformly w.r.t $y$ on any bounded subset.
In this case $f,g$ are jointly quasi-periodic, so \eqref{ex1} admits a quasi-periodic in distribution solution.
\end{example}

\begin{example}\rm
Consider the stochastic heat equation on the
interval [0,1] with Dirichlet boundary condition:
\begin{align}\label{ex2}
\frac{\partial u}{\partial t} =& \frac{\partial^2 u}{\partial
\xi^2}+ \frac{(\sin t+\cos\sqrt{3}t)\sin u}{3}\\
&
\quad + \frac{u}{u^2+1}\cdot\cos \Big{(}\frac{1}{2+\sin t +\sin \sqrt2\, t}\Big{)} \frac{\partial W}{\partial t}\nonumber \\
=&: \frac{\partial^2 u}{\partial \xi^2} + f(t,u) + g(t,u)\frac{\partial W}{\partial t}, \nonumber\\
u(t,0)=& u(t,1)=0,\;\;\; t>0.  \nonumber
\end{align}
Here $W$ is a one-dimensional two-sided Brownian motion. Let $A$ be the
Laplace operator, then $A: D(A)=H^2(0,1) \cap H^1_0 (0,1)\to
L^2(0,1)$. Denote $H := L^2(0,1)$ and the norm on $H$ by $||\cdot||$. Then the stochastic
heat equation can be written as an abstract evolution equation
\begin{align}\label{aeex}
d Y(t) &=({A}  Y(t)+ {F}(t,Y(t)))d t + {G}(t, Y(t))d W(t)
\end{align}
on the Hilbert space $H$, where
\begin{align*}
Y(t) :=& u(t,\cdot), \quad {F}(t,Y(t)):= f(t, u(t,\cdot)), \quad {G}(t,Y(t)):= g(t, u(t,\cdot)).
\end{align*}
Note that, the operator $A$ has eigenvalues
$\{-n^2\pi^2\}_{n=1}^\infty$ and generates a $\mathcal
C^0$-semigroup $T(t)$ on $H$ satisfying $||T(t)||\le
e^{-\pi^2 t}$ for $t\ge 0$, i.e. $\mathcal N=1$ and $\nu=\pi^2$. Note that
$\max\{Lip(F),Lip(G)\}\le 1$, so it is immediate to verify that conditions (C1)-(C2) hold and the restrictions on Lipschitz constant in
Theorems \ref{t3.4.4}, \ref{thC1} and \ref{thC2} are satisfied. As pointed out in Remark \ref{remL10}-(iii), we need to check condition (C3).
Indeed, since $f$ is bounded, for given $\alpha>0$ we have
\begin{equation}\label{exui}
\sup_{t\in\mathbb R, ||u||\le M}\int_{[0,1]}|f(t,u(x))|^{2+\alpha} d x <\infty
\end{equation}
for any $M>0$, i.e. the family $\{|f(t,u(x))|^2: t\in \mathbb R, ||u||\le M\}$ of functions of $x$ is uniformly integrable on $[0,1]$.
This implies that for $t_n\to t$, by choosing $k$ large enough,
\begin{align*}
&\int_{[0,1]} |f(t_n, u(x)) - f(t,u(x))|^2 d x \\
&\le \int_{[0,1]\cap \mathcal M_k}   |f(t_n, u(x)) - f(t,u(x))|^2 d x +
\int_{[0,1]\setminus \mathcal M_k}   |f(t_n, u(x)) - f(t,u(x))|^2 d x
\end{align*}
is sufficiently small, where $\mathcal M_k:=\{x\in[0,1]: |u(x)|\le k\}$. That is, (C3) holds.

Finally note that $F$ is quasi-periodic in $t$ and $G$ is Levitan
almost periodic in $t$, uniformly w.r.t. $Y\in H$.


By Theorem \ref{t3.4.4}, \eqref{aeex} (and hence \eqref{ex2}) admits a unique $L^2(\mathbb P,H)$-bounded mild solution,
and by Corollary \ref{corAA1} this unique bounded solution is Levitan almost periodic
in distribution. By Theorem \ref{thC2}, this bounded solution is globally asymptotically stable in square-mean sense.
By Corollary \ref{UB}, all the solutions of \eqref{ex2} with $L^2$-initial value are bounded by a constant after
sufficiently long time.

\end{example}

\begin{remark}\rm
As pointed out in Remark \ref{remL10}, to apply our results for stochastic PDEs, we need to check the condition (C3), which is not
easy to check in some situations.  We will try to weaken or remove this condition in our future work.
\end{remark}

\section*{Acknowledgements}

This work is partially supported by NSFC Grants 11271151, 11522104, and the startup and
Xinghai Youqing funds from Dalian University of Technology.


\begin{thebibliography}{99}




\bibitem{AT}
L. Arnold and C. Tudor, Stationary and almost periodic solutions of
almost periodic affine stochastic differential equations, {\it
Stochastics Stochastics Rep.} {\bf 64}  (1998),  177--193.

\bibitem{BG}
B. Basit and H. Gnzler, Spectral criteria for solutions of evolution equations and comments
on reduced spectra, {\it arXiv preprint} (2010), arXiv:1006.2169.

\bibitem{Beb_1940}
V. M. Bebutov,
\newblock On the shift dynamical systems on the space of continuous functions,
\newblock{\em Bull. of Inst. of Math. of Moscow University} 2;5 (1940), pp.1-65. (in Russian)



\bibitem{BD}
P. H. Bezandry and T. Diagana, Existence of almost periodic solutions
to some stochastic differential equations, {\it Appl. Anal.}  {\bf
86} (2007), 819--827.



\bibitem{Bir_1927}
G. D. Birkhoff,
\newblock{\em Dynamical Systems.}
\newblock Amer. Math. Soc. Colloq. Publ., vol.IX,
American Mathematical Society, Providence, RI, 1927.

\bibitem{Boh_1955}
S. Bochner,
\newblock Curvature and Betti numbers in real and complex vector bundles,
\newblock{\em Univ. e Politec. Torino.Rend. Sem. Mat.} {\bf 15} (1955--56), 225--253.

\bibitem{B62}
S. Bochner, A new approach to almost periodicity. \it Proc. Nat.
Acad. Sci. U.S.A. \bf 48 \rm (1962), 2039--2043.

\bibitem{Bol_1893}
P. Bohl, \newblock{\em ber die Darstellung von Funktionen einer Variabeln durch trigonometrische
Reihen mit mehreren einer Variabeln proportionalen Argumenten}.
\newblock Magisterdissertation, Dorpat, 1893, 31pp.

\bibitem{Bol_1906}
P. Bohl, \newblock \"Uber eine Differentialgleichung der St\"orungstheorie (in German),
\newblock {\em J. Reine Angew. Math.} {\bf 131} (1906),  268--321.


\bibitem{Bohr_1923}
H. Bohr,
\newblock Sur les Fonction Presque-Periodiques,
\newblock{\em C. R. Acad. Sci. Paris} {\bf 177} (1923), 737--739.

\bibitem{Bohr_I}
H. Bohr,
\newblock Zur Theorie der Fastperiodischen Funktionen mit
Funktionen. I. Eine Verallgemeinerung der Theorie der
Fourerreinhen,
\newblock{\em Acta. Math.} {\bf 45} (1924), 29--127.

\bibitem{Bohr_II}
H. Bohr,
\newblock Zur Theorie der Fastperiodischen Funktionen mit
Funktionen. II. Zusammenhang der Fastperiodischen Funktionen mit
Funktionen von Unendlich Vielen Variablen; Gleichm\"{a}ssige
Approximation durch Trigonometrische Summen,
\newblock{\em Acta. Math.} {\bf 46}(1925), 101--214.

\bibitem{Bohr_III}
H. Bohr, \newblock Zur Theorie der Fastperiodischen Funktionen mit
Funktionen. III. Dirichletentwicklung analytischer Funktionen,
\newblock{\em Acta. Math.} {\bf 47} (1926),  237--281.

\bibitem{Bohr_I1947}
H. Bohr,
\newblock{\em Almost Periodic Functions.}
\newblock Chelsea Publishing Company, New York, 1947. ii+114 pp.


\bibitem{bro75}
I. U. Bronsteyn,
\newblock {\em Extensions of Minimal Transformation Group.}
\newblock{\em Kishinev,} Stiintsa, 1974 (in Russian). [English
translation: Extensions of Minimal Transformation Group, Sijthoff
\& Noordhoff, Alphen aan den Rijn, 1979]


\bibitem{CC_I}
T. Caraballo and D. Cheban,
\newblock Almost periodic and almost automorphic solutions
of linear differential/difference equations without Favard's
separation condition. I,
\newblock{\em J. Differential Equations}, {\bf 246} (2009), 108--128.

\bibitem{CC_II}
T. Caraballo and D. Cheban,
\newblock Almost periodic and almost automorphic solutions
of linear differential/difference equations without Favard's
separation condition. II,
\newblock{\em J. Differential Equations} {\bf 246} (2009), 1164--1186.

\bibitem{CC_2011}
T. Caraballo and D. Cheban,
\newblock Levitan/Bohr almost periodic and almost automorphic
solutions of second-order monotone differential equations,
\newblock{\em J. Differential Equations}  {\bf 251} (2011), 708--727.


\bibitem{CC_2013_1}
T. Caraballo and D. Cheban,
\newblock Almost periodic and almost automorphic solutions of linear
differential equations,
\newblock{\em Discrete Contin. Dyn. Syst.}  {\bf 33} (2013), 1857--882.


\bibitem{Che_2008}
D. Cheban,
\newblock Levitan Almost periodic and almost automorphic solutions of
$V$-monotone differential equations,
\newblock{\em J. Dynam. Differential Equations}  {\bf 20} (2008), 69--697.

\bibitem{Che_2009}
D. Cheban,
\newblock{\em Asymptotically Almost Periodic Solutions of Differential
Equations.} \newblock Hindawi Publishing Corporation, New York,
2009, ix+186 pp.


\bibitem{Ch2015}
D. Cheban, {\em Global attractors of non-autonomous dissipative dynamical systems}. Interdisciplinary Mathematical Sciences, vol. 1.
World Scientific Publishing Co. Pte. Ltd., Hackensack, NJ, 2004. xxiv+502 pp.


\bibitem{CM_2005}
D. Cheban and C. Mammana,
\newblock Invariant manifolds, almost periodic and almost automorphic
solutions of seconde-order monotone equations,
\newblock{\em Int. J. Evol. Equ.} \bf 1 \rm (2005), 319--343.

\bibitem{CS_2008}
D. Cheban and B. Schmalfuss,
\newblock Invariant manifolds, global attractors,
almost automrphic and almost periodic solutions of non-autonomous
differential equations,
\newblock{\em J. Math. Anal. Appl.} {\bf340} (2008), 374--393.

\bibitem{CZ}
Z. Chen and W. Lin, Square-mean weighted pseudo almost automorphic solutions for non-autonomous stochastic evolution equations,
{\it J. Math. Pures Appl.} \bf 100 \rm (2013), 476--504.





\bibitem{DTr}
G. Da Prato and L. Tubaro, Some results on periodic
measures for differential stochastic equations with additive noise,
{\it Dynam. Systems Appl.} {\bf 1} (1992), 103--120.

\bibitem{DT}
G. Da Prato and C. Tudor, Periodic and almost periodic solutions for
semilinear stochastic equations.  {\it Stoch. Anal. Appl.} {\bf
13} (1995), 13--33.

\bibitem{DZ}
G. Da Prato and J. Zabczyk, {\em Stochastic Equations in Infinite Dimensions}. Second edition.
Encyclopedia of Mathematics and its Applications, 152. Cambridge University Press, Cambridge, 2014. xviii+493 pp.

\bibitem{Dal}
Yu. L. Daletskii and M. G. Krein,
\newblock {\em Stability of Solutions of Differential Equations in Banach
Space}.
\newblock Moscow, "Nauka", 1970. [English transl., Amer. Math. Soc., Providence, RI
1974.]









\bibitem{FL}
M. Fu and Z. Liu, Square-mean almost automorphic solutions for some
stochastic differential equations, {\it Proc. Amer. Math. Soc.} {\bf
138} (2010), 3689--3701.

\bibitem{Hal}
A. Halanay,  Periodic and almost periodic solutions to affine
stochastic systems.  Proceedings of the Eleventh International
Conference on Nonlinear Oscillations (Budapest, 1987),  94--101,
J\'anos Bolyai Math. Soc., Budapest, 1987.


\bibitem{Ich}
A. Ichikawa, Bounded solutions and periodic solutions of a linear stochastic evolution equation. Probability theory and
mathematical statistics (Kyoto, 1986), 124--130, Lecture Notes in Math., 1299, Springer, Berlin, 1988.

\bibitem{KMR}
M. Kamenskii, O. Mellah, and P. Raynaud de Fitte,
Weak averaging of semilinear stochastic differential equations with almost periodic coefficients,
\it J. Math. Anal. Appl. \bf427 \rm(2015), 336--364.

\bibitem{Kh}
R.  Khasminskii, {\it Stochastic Stability of Differential Equations}. Translated from the Russian by D. Louvish. Monographs and
Textbooks on Mechanics of Solids and Fluids: Mechanics and Analysis, 7. Sijthoff \& Noordhoff, Alphen aan den Rijn-Germantown, Md., 1980. xvi+344 pp.

\bibitem{Lev_1938}
B. Levitan,
\newblock ber eine Verallgemeinerung der stetigen fastperiodischen Funktionen von H. Bohr.
\newblock{\em  Ann. of Math. (2)}  {\bf 40} (1939), 805--815. (in German)

\bibitem{Lev_1953}
B. M. Levitan,
\newblock{\em Almost Periodic Functions.}
\newblock Gosudarstv. Izdat. Tekhn-Teor. Lit., Moscow, 1953, 396 pp. (in
Russian)

\bibitem{Lev-Zhi}
B. M. Levitan  and V. V. Zhikov,
\newblock{\em Almost Periodic Functions and Differential
Equations.}
\newblock Moscow State University Press, Moscow, 1978, 204 pp. (in Russian).
[English translation: Almost Periodic Functions and Differential
Equations. Cambridge Univ. Press, Cambridge, 1982, xi+211 pp.]

\bibitem{LS}
Z. Liu and K. Sun, Almost automorphic solutions for stochastic differential equations
driven by L\'{e}vy noise, {\it J. Funct. Anal.}  {\bf 226}
(2014), 1115--1149.

\bibitem{LW_2016}
Z. Liu and W. Wang,
\newblock Favard separation method for almost periodic stochastic differential equations,
\newblock{\em J. Differential Equtions} {\bf 260} (2016), 8109--8136.

\bibitem{Mo}
T. Morozan,  Periodic solutions of affine stochastic differential equations, {\it Stochastic Anal. Appl.}
{\bf 4} (1986), 87--110.






\bibitem{Sel_71}
G. R. Sell,
\newblock {\em Lectures on Topological Dynamics and Differential Equations},
Vol. 2 of {\em Van Nostrand Reinhold math. studies}.
\newblock Van Nostrand--Reinbold, London, 1971.


\bibitem{Sch68} B. A. Shcherbakov,
\newblock A certain class of Poisson stable solutions
of differential equations,
\newblock{\it Differentsial'nye Uravneniya} {\bf 4} (1968), no.2, 238--243. (in Russian)



\bibitem{Sch72}
B. A. Shcherbakov,
\newblock {\em Topologic Dynamics and Poisson Stability of Solutions of
Differential Equations}.
\newblock \c{S}tiin\c{t}a, Chi\c{s}in\u{a}u, 1972, 231 pp. (in Russian)

\bibitem{scher75}
B. A. Shcherbakov ,
\newblock The comparability of the motions of dynamical systems with
regard to the nature of their recurrence,
\newblock{\it Differentsial'nye Uravneniya} {\bf 11} (1975), no. 7, 1246--1255. (in Russian)
[English translation: {\it Differential Equations} {\bf 11} (1975), no.7,
937--943].

\bibitem{Sch85}
B. A. Shcherbakov,
\newblock {\em Poisson Stability of Motions of Dynamical Systems and Solutions
of Differential Equations}.
\newblock \c{S}tiin\c{t}a, Chi\c{s}in\u{a}u, 1985, 147 pp. (in Russian)

\bibitem{scher-ch}
B. A. Shcherbakov and D. Cheban,
\newblock Asymptotically Poisson stable motions of dynamical systems and
comparability of their reccurence in limit,
\newblock {\it Differentsial'nye Uravneniya} {\bf 13} (1977), no. 5, 898--906. (in Russian)
[English translation: {\it Differential Equations} {\bf 13} (1978), no. 5, 618--624]

\bibitem{ShFa_1977}
B. A. Shcherbakov and N. S. Fal'ko,
\newblock The minimality of sets and the Poisson stability of motions in
homomorphic dynamical systems,
\newblock{\em Differencial'nye
Uravnenija,} {\bf 13} (1977), no. 6, 1091--1097.(in Russian)
[English translation: {\it Differential Equations}
{\bf 13} (1978), no. 6, 755--758]

\bibitem{SY}
W. Shen and Y. Yi, Almost automorphic and almost periodic dynamics in skew-product semiflows, \it Mem. Amer. Math. Soc.
\bf136 \rm (1998), no. 647, x+93pp.

\bibitem{sib}
K. S. Sibirsky,
\newblock {\em Introduction to Topological Dynamics.\/}
\newblock Kishinev, RIA AN MSSR, 1970, 144 p. (in Russian). [English
translationn: Introduction to Topological Dynamics. Noordhoff,
Leyden, 1975]

\bibitem{T}
C. Tudor, Almost periodic solutions of affine stochastic evolution
equations,  {\it Stochastics Stochastics Rep.} {\bf 38} (1992),
251--266.

\bibitem{WL}
Y. Wang and Z. Liu, Almost periodic solutions for stochastic
differential equations with L\'evy noise,  {\it Nonlinearity} {\bf
25} (2012), 2803--2821.

\bibitem{Wee_1965} W. A. Veech,
\newblock Almost automorphic functions on groups,
\newblock{\it Amer. J. Math.}  {\bf 87} (1965), 719--751.

\bibitem{X}
Z. Xia, Pseudo almost automorphic in distribution solutions of semilinear stochastic integro-differential
equations by measure theory, \it Internat. J. Math. \bf 26 \rm(2015), no. 13, 1550112, 24 pp.




\end{thebibliography}
\end{document}